\documentclass[a4paper,11pt,reqno,twoside]{article}
\usepackage{amssymb}
\usepackage[frenchb,english]{babel}
\usepackage{amscd}
\usepackage{amsmath,amsthm,amsfonts,amssymb,graphicx}
\usepackage[colorlinks,linkcolor=blue,anchorcolor=blue,citecolor=green]{hyperref}
\usepackage{mathrsfs}
\usepackage{fancyhdr}
\usepackage{multicol}
\usepackage{abstract}

\makeatother
\begin{document}
\theoremstyle{plain}
\newtheorem{Definition}{Definition}[section]
\newtheorem{Proposition}{Proposition}[section]
\newtheorem{Property}{Property}[section]
\newtheorem{Theorem}{Theorem}[section]
\newtheorem{Lemma}[Theorem]{\hspace{0em}\bf{Lemma}}
\newtheorem{Corollary}[Theorem]{Corollary}
\newtheorem{Remark}{Remark}[section]
%%%%%%%%%%%%%%%%%%%%%%%%%%%%%%%%%%%%%%%%%%%%%%%%%%%%%%%%%%%%%%%%

%%%%%%%%%%%%%%%%%%%%%%%%%%%%%%%%%%%%%%%%%%%%%%%%%%%%%%%%%%%%%%%%
\setlength{\oddsidemargin}{ 1cm}  % 3.17cm - 1 inch
\setlength{\evensidemargin}{\oddsidemargin}
\setlength{\textwidth}{13.50cm}
\vspace{-.8cm}

\noindent  {\LARGE Balanced  metrics on some Hartogs type domains over bounded
symmetric domains}\\\\
\noindent\text{Zhiming Feng  }\\
\noindent\small {School of Mathematical and Information Sciences, Leshan Normal University, Leshan, Sichuan 614000, P.R. China } \\
\noindent\text{Email: fengzm2008@163.com}

\vskip 5pt
\noindent\text{Zhenhan Tu$^{*}$ }\\
\noindent\small {School of Mathematics and Statistics, Wuhan
University, Wuhan, Hubei 430072, P.R. China} \\
\noindent\text{Email: zhhtu.math@whu.edu.cn
}
\renewcommand{\thefootnote}{{}}
\footnote{\hskip -16pt {$^{*}$Corresponding author. \\}}
\\

\normalsize \noindent\textbf{Abstract}\quad {The definition of balanced metrics was originally given by Donaldson
in the case of a compact polarized K\"{a}hler manifold in 2001, who
also established the existence of such metrics on any compact
projective K\"{a}hler manifold with constant scalar curvature.
Currently, the only noncompact manifolds on which
 balanced metrics are known to exist are homogeneous domains.
The generalized Cartan-Hartogs domain
$\big(\prod_{j=1}^k\Omega_j\big)^{{\mathbb{B}}^{d_0}}(\mu)$ is
defined as the Hartogs type domain constructed over the product
$\prod_{j=1}^k\Omega_j$ of irreducible bounded symmetric domains
$\Omega_j$ $(1\leq j \leq k)$, with the fiber over each point
$(z_1,\cdots,z_k)\in \prod_{j=1}^k\Omega_j$ being a ball in
$\mathbb{C}^{d_0}$ of the radius
$\prod_{j=1}^kN_{\Omega_j}(z_j,\overline{z_j})^{\frac{\mu_j}{2}}$ of the
product of positive powers of their generic norms. Any such domain
$\big(\prod_{j=1}^k\Omega_j\big)^{{\mathbb{B}}^{d_0}}(\mu)$ $(k\geq
2)$ is a bounded nonhomogeneous domain. The purpose of this paper is
to obtain necessary and sufficient conditions for the metric $\alpha
g(\mu)$ $(\alpha>0)$ on the domain
$\big(\prod_{j=1}^k\Omega_j\big)^{{\mathbb{B}}^{d_0}}(\mu)$ to be a
balanced metric, where $g(\mu)$ is its canonical metric. As the main
contribution of this paper, we obtain the existence of balanced
metrics for a class of such bounded nonhomogeneous domains.
\smallskip\\\
\textperiodcentered \; Cartan-Hartogs domains \textperiodcentered \;
K\"{a}hler metrics
\smallskip\\\
  \textperiodcentered \, 32M15  \textperiodcentered \, 32Q15
%%%%%%%%%%%%%%%%%%%%%%%%%%%%%%%%%%%%%%%%%%%%%%%%%%%%%%%%%%%%%%%%

%%%%%%%%%%%%%%%%%%%%%%%%%%%%%%%%%%%%%%%%%%%%%%%%%%%%%%%%%%%%%%%%
\setlength{\oddsidemargin}{-.5cm}  % 3.17cm - 1 inch
\setlength{\evensidemargin}{\oddsidemargin}
\pagenumbering{arabic}
\renewcommand{\theequation}
{\arabic{section}.\arabic{equation}}
%%%%%%%%%%%%%%%%%%%%%%%%%%%%%%%%%%%%%%%%%%%%%%%%%%%%%%%%%%%%%%%%
\setcounter{section}{0}
 \setcounter{equation}{0}
\section{Induction}

The expansion of the Bergman kernel has received a lot of attention
recently, due to the influential work of Donaldson, see e.g. \cite{Donaldson},
about the existence and uniqueness of constant
scalar curvature K\"{a}hler metrics (cscK metrics). Donaldson used
the asymptotics of the Bergman kernel proved by Catlin \cite{Cat}
and Zelditch \cite{Zeld} and the calculation of Lu \cite{Lu} of the
first coefficient in the expansion to give conditions for the
existence of cscK metrics. This work inspired many papers on the
subject since then. For the reference of the expansion of the
Bergman kernel, see also Engli\v{s} \cite{E2}, Loi \cite{Lo},
Ma-Marinescu \cite{MM07, MM08, MM12}, Xu \cite{X}
 and references therein.

Assume that $D$ is a bounded domain in $\mathbb{C}^n$ and $\varphi$
is a strictly plurisubharmonic function on $D$. Let $g$ be a
K\"{a}hler  metric  on $D$ associated to the K\"{a}hler form
$\omega=\frac{\sqrt{-1}}{2\pi}\partial\overline{\partial}\varphi$.
For $\alpha>0$, let $\mathcal{H}_{\alpha}$ be the weighted Hilbert
space of square integrable holomorphic functions on $(D, g)$ with
the weight $\exp\{-\alpha \varphi\}$, that is,
$$\mathcal{H}_{\alpha}:=\left\{ f\in \textmd{Hol}(D): \int_{D}|f|^2\exp\{-\alpha \varphi\}\frac{\omega^n}{n!}<+\infty\right\},$$
where $\textmd{Hol}(D)$ denotes the space of holomorphic functions
on $D$. Let $K_{\alpha}(z,\overline{z})$ be the Bergman kernel
(namely, the reproducing kernel) of the Hilbert space
$\mathcal{H}_{\alpha}$ if $\mathcal{H}_{\alpha}\neq \{0\}$. The
Rawnsley's $\varepsilon$-function on $D$ associated to the metric
$g$  is defined by
\begin{equation}\label{eq1.4}
 \varepsilon_{\alpha}(z):=\exp\{-\alpha \varphi(z)\}K_{\alpha}(z,\overline{z}),\;\; z\in D.
\end{equation}
Note the Rawnsley's $\varepsilon$-function depends only on the
metric $g$ and not on the choice of the K\"{a}hler potential
$\varphi$ (which is defined up to an addition with the real part of
a holomorphic function on $D$). The function
$\varepsilon_{\alpha}(z)$ has appeared in the literature under
different names. The earliest one was probably the $\eta$-function
of Rawnsley \cite{Ra} (later renamed to $\varepsilon$-function in
Cahen-Gutt-Rawnsley \cite{CGR}) defined for arbitrary K\"ahler
manifolds, followed by the distortion function of Kempf \cite{Ke}
and Ji \cite{Ji} for the special case of Abelian varieties, and of
Zhang \cite{Zhang} for complex projective varieties.

The asymptotics of the Rawnsley's $\varepsilon$-function $
\varepsilon_{\alpha}$ was expressed in terms of the parameter
$\alpha$ for compact manifolds by Catlin \cite{Cat} and Zelditch
\cite{Zeld} (for $\alpha\in \mathbb{N}$) and for non-compact
manifolds by Ma-Marinescu \cite{MM07,MM08}. In some particular case
it was also proved by Engli\v{s} \cite{E1,E2}.

\begin{Definition} The metric $\alpha g$ on $D$ is balanced  if the Rawnsley's
$\varepsilon$-function $\varepsilon_{\alpha}(z)$ $(z\in D)$ is a
positive constant on $D$.
\end{Definition}

The metric for which the function $\varepsilon_{\alpha}(z)$ is
constant was also called critical metric in Zhang \cite{Zhang}. The
definition of balanced metrics was originally given by Donaldson
\cite{Donaldson} in the case of a compact polarized K\"{a}hler
manifold $(M,g)$ in 2001, who also established the existence of such
metrics on any (compact) projective K\"{a}hler manifold with
constant scalar curvature. The definition of balanced metrics was
generalized in Arezzo-Loi
 \cite{Are-Loi} and Engli\v{s} \cite{E3} to the noncompact case. We give only
the definition for those K\"{a}hler metrics which admit globally
defined potentials in this paper.

Let $g$ be a K\"{a}hler  metric on a bounded domain $D$ in
$\mathbb{C}^n$ associated to the K\"{a}hler form
$\omega=\frac{\sqrt{-1}}{2\pi}\partial\overline{\partial}\varphi$
for a strictly plurisubharmonic function $\varphi$ on $D$. Let
$\mathcal{H}_{\alpha}$ ($\alpha>0$) be the weighted Hilbert space of
square integrable holomorphic functions on $(D, g)$ with the weight
$\exp\{-\alpha \varphi\}$. Let $K_{\alpha}(z,\bar z)$ be the
reproducing kernel of $\mathcal{H}_{\alpha}$ if
$\mathcal{H}_{\alpha}\neq \{0\}$. Let $h$ be a K\"{a}hler  metric on
$D$ associated to the K\"{a}hler form
$\frac{\sqrt{-1}}{2\pi}\partial\overline{\partial}\ln
K_{\alpha}(z,\bar z)$. It is again readily seen that
$\frac{\sqrt{-1}}{2\pi}\partial\overline{\partial}\ln
K_{\alpha}(z,\bar z)$ is independent of the choice of the potential
$\varphi$ and thus is uniquely determined by the original metric
$g$. So, by definition, we have
\begin{equation*}
\frac{\sqrt{-1}}{2\pi}\partial\overline{\partial}\ln
\varepsilon_{\alpha}(z)+ \alpha \cdot
\frac{\sqrt{-1}}{2\pi}\partial\overline{\partial} \varphi(z) =
\frac{\sqrt{-1}}{2\pi}\partial\overline{\partial}\ln
K_{\alpha}(z,\overline{z}),\;\; z\in D.
\end{equation*}
Therefore, if the metric $\alpha g$ is balanced on $D$, then we have
$\alpha g =h $ on $D$.

Recall that the Fubini-Study metric $ds_{FS}$ on the $N$-dimensional
complex projective space $P^N(\mathbb{C})$ ($1\leq N \leq +\infty$)
is this metric whose K\"{a}hler form $\omega_{FS}$  in homogeneous
coordinates $Z_0, Z_1, Z_2,\cdots$ is given by $\omega_{FS}
=\frac{\sqrt{-1}}{2\pi}\partial\overline{\partial}\ln
(\sum_{j=0}^{N}|Z_j|^2).$ A balanced metric $\alpha g$ on $D$ can be
also viewed as a particular projectively induced K\"{a}hler metric
for which the K\"{a}hler immersion
$$
f: D\rightarrow  P^N(\mathbb{C}), \;\;  z \longmapsto
[f_0(z),f_1(z),\cdots, f_j(z),\cdots], $$ where $N+1$ $(0\leq N \leq
+\infty)$ denotes the complex dimension of $\mathcal{H}_{\alpha}$,
is given by an orthonormal basis $\{f_j\}$ of the Hilbert space
$\mathcal{H}_{\alpha}$ if $\mathcal{H}_{\alpha}\neq \{0\}$. Indeed,
the map $f: D\rightarrow P^N(\mathbb{C})$ is well-defined since
$\varepsilon_{\alpha}$ is a positive constant and hence for any
$z\in D$ there exists some $f_j$ such that $f_j(z)\not= 0.$
Moreover,
\begin{align*}
f^*\omega_{FS}&
=\frac{\sqrt{-1}}{2\pi}\partial\overline{\partial}\ln
\sum_{j=0}^{N}|f_j(z)|^2  \\
&=\frac{\sqrt{-1}}{2\pi}\partial\overline{\partial}\ln
K_{\alpha}(z,\bar z)  \\
&=\frac{\sqrt{-1}}{2\pi}\partial\overline{\partial}\ln
\varepsilon_{\alpha}(z)+\alpha\cdot
\frac{\sqrt{-1}}{2\pi}\partial\overline{\partial}\varphi(z).
\end{align*}
Hence, if the metric $\alpha g$ on $D$ is balanced, then the map $f$
is a holomorphically isometric mapping from $(D, \alpha g)$ into
$(P^N(\mathbb{C}), ds_{FS})$ (Note a projectively induced metric is
not always balanced, e.g., see Example 1 in Loi-Zedda \cite{LZ}).
The map $f$ is called in Cahen-Gutt-Rawnsley \cite{CGR} the coherent
states map. Balanced metric plays a fundamental role in the geometric
quantization and quantization by deformation of a K\"{a}hler
manifold (e.g., see Berezin \cite{Berezin}, Cahen-Gutt-Rawnsley \cite{CGR}, Engli\v{s} \cite{E0} and Lui\'c
\cite{Lui}). It also related to the Bergman kernel expansion (e.g.,
see Loi \cite{Lo} and references therein).

We remark that in Donaldson's study (e.g., see Donaldson
\cite{Donaldson}) of asymptotic stability for polarized algebraic
manifolds, balanced metrics play a central role when the polarized algebraic manifolds admit K\"{a}hler metrics of constant scalar
curvature. The balanced metrics are also of significance in some
questions concerning (semi) stability of projective algebraic
varieties (see e.g., Mabuchi \cite{Mabuchi} and Zhang \cite{Zhang}).
For the study of the balanced metrics,  see also Cuccu-Loi
\cite{Cuc-Loi},  Engli\v{s} \cite{E0,E3},  Feng-Tu \cite{FT}, Greco-Loi
\cite{Gre-Loi},  Loi \cite{Loi2}, Loi-Mossa \cite{Loi-Mossa}, Loi-Zedda \cite{Loi-Zed, LZ},
Loi-Zedda-Zuddas \cite{Loi-Zed-Zud} and Zedda \cite{Zed}.

However, very little seems to be known about the existence of
balanced metrics on general noncompact manifolds or even domains in
$\mathbb{C}^n$ (e.g., see Engli\v{s} \cite{E3} and Loi-Mossa \cite{Loi-Mossa}). Currently, the only
noncompact manifolds on which balanced metrics are known to exist are the homogeneous spaces.
 In the nonhomogeneous setting, the problem of
existence of balanced metrics seems to be open.

In this paper we will obtain the existence of balanced metrics on a
class of bounded nonhomogeneous domains.

\vskip 5pt

Every Hermitian symmetric manifold of noncompact type can be
realized as a bounded symmetric domain in some $\mathbb{C}^d$ by the
Harish-Chandra embedding theorem. Then, the classification of
Hermitian symmetric manifolds of noncompact type coincides with the
classification of bounded symmetric domains. In 1935, E. Cartan
proved that there exist only six types of irreducible bounded
symmetric domains. They are four types of classical bounded
symmetric domains and two exceptional domains.

Let $\mathcal{M}_{m,n}$ be the set of all $m\times n$ matrices
$z=(z_{ij})$ with complex entries. Let ${\overline z}$ be the
complex conjugate of the matrix $z$ and let ${z}^t$ be the transpose
of the matrix $z$. $I$ denotes the identity matrix. If a square
matrix $z$ is positive definite, then we write $z>0$. For each
irreducible bounded symmetric domain $\Omega$ (refer to Hua
\cite{Hua}), we list the generic norm $N_\Omega(z,\overline{z})$ of
$\Omega$ according to its type as following.

$(i)$ If $\Omega=\Omega_I(m,n):=\{z\in \mathcal{M}_{m,n}:
I-z{\overline z}^t>0 \} \subset \mathbb{C}^d$ ($1\leq m\leq n,\;
d=mn$) (the classical domains of type $I$), then
$N_\Omega(z,\overline{z})=\det(I-z{\overline z}^t).$ Specially, when
$m=1$, then $\Omega$ is the unit ball $\mathbb{B}^n$ in $
\mathbb{C}^{n}$ and $N_\Omega(z,\overline{z})=1-\|z\|^2.$

$(ii)$ If $\Omega= \Omega_{II}(n):=\{z\in \mathcal{M}_{n,n}: z^t=-z,
I-z{\overline z}^t>0 \}\subset \mathbb{C}^d$ ($n\geq 5,\;
d={n(n-1)}/{2}$) (the classical domains of type $II$), then
$N_\Omega(z,\overline{z})=(\det(I-z{\overline z}^t))^{1/2} .$

$(iii)$ If $\Omega= \Omega_{III}(n):=\{z\in \mathcal{M}_{n,n}:
z^t=z, I-z{\overline z}^t>0 \}\subset \mathbb{C}^d$ ($n\geq 2,\;
d=n(n+1)/2$) (the classical domains of type $III$), then
$N_\Omega(z,\overline{z})=\det(I-z{\overline z}^t).$

$(iv)$ If $\Omega=\Omega_{IV}(n):=\{z\in \mathbb{C}^{n}:
1-2z{\overline z}^t+|zz^t|^2>0,  z{\overline z}^t<1\} $ ($n\geq 5$)
(the classical domains of type $IV$), then
$N_\Omega(z,\overline{z})=1-2z{\overline z}^t +|zz^t|^2.$

There are two more exceptional domains  $\Omega_{\textrm{V}}(16)$
and  $\Omega_{\textrm{VI}}(27)$
 of dimension 16 and 27 respectively. We call them Type V and Type VI,
respectively. For the precise definitions of these exceptional
domains, see Part V in \cite{FKKLR}.

Now we introduce the numerical invariants: the rank $r$,  the
multiplicities $a$ and $b$, and the genus $p = 2 + a(r-1)+ b$ for an
irreducible bounded symmetric domain. The list of numerical
invariants of each irreducible bounded symmetric domain of six types
is the following:

(1) For $\Omega_I(m,n)\; (1\leq m\leq n)$, its rank $r=m$, its
multiplicities $a=2$ and $b=n-m$,  and its genus $p=m+n$.

(2) For $\Omega_{II}(2n)\; (n\geq 3) $, its rank $r=n$, its
multiplicities $a=4$ and $b=0$,  and its genus $p=2(2n-1)$;
 For $\Omega_{II}(2n+1)\;(n\geq 2) $, its rank $r=n$, its
multiplicities $a=4$ and $b=2$,  and its genus $p=4n$.

(3) For $\Omega_{III}(n)\;(n\geq 2)$,  its rank $r=n$, its
multiplicities $a=1$ and $b=0$,  and its genus $p=n+1$.

(4) For $\Omega_{IV}(n)\;(n\geq 5)$, its rank $r=2$, its
multiplicities $a=n-2$ and $b=0$,  and its genus $p=n$.

(5) For $\Omega_{\textrm{V}}(16)$, its rank $r=2$, its
multiplicities $a=6$ and $b=4$, and its genus $p=12$.

(6) For  $\Omega_{\textrm{VI}}(27)$, its rank $r=3$, its
multiplicities $a=8$ and $b=0$,  and its genus $p=18$.

Remark that $\Omega_{II}(2),\Omega_{III}(1)$ and $\Omega_{IV}(1)$
are biholomorphically equivalent to $\Omega_I(1,1)$;
$\Omega_{II}(3)$ is biholomorphically equivalent to $\Omega_I(3,1)$;
$\Omega_{II}(4)$ is biholomorphically equivalent to
$\Omega_{IV}(6)$; $\Omega_{IV}(3)$ is biholomorphically equivalent
to $\Omega_{III}(2)$, and $\Omega_{IV}(4)$ is biholomorphically
equivalent to $\Omega_I(2,2)$. But $\Omega_{IV}(2)$ is not an
irreducible bounded symmetric domain, since $\Omega_{IV}(2)$ is
biholomorphically equivalent to the bidisc $\mathbb{B}\times
\mathbb{B}$.

Throughout this paper, each irreducible bounded symmetric domain is
always one of irreducible bounded symmetric domains of six types.

\vskip 5pt

Let $\Omega$ be an irreducible bounded symmetric domain in
$\mathbb{C}^d$ of genus $p$ in its Harish-Chandra realization. Since
$\Omega$ is a bounded circular domain and contains the origin, there
is a homogeneous holomorphic polynomial set
$$\left\{\frac{1}{\sqrt{V(\Omega)}},h_1(z),h_2(z),\cdots \right\},$$
such that it is an orthonormal basis of the Hilbert space
$A^2(\Omega)$ of square-integrable holomorphic functions on
$\Omega$, where $V(\Omega )$ is the Euclidean volume of $\Omega$ in
$\mathbb{C}^d$ and $\deg h_j(z)\geq 1$ (so $h_j(0)=0$) for
$j=1,2,\cdots$. Then the Bergman kernel $K_{\Omega}(z,\bar{\xi})$ of
$\Omega$ is given by
\begin{equation*}
K_{\Omega}(z,\bar{\xi})=\frac{1}{V(\Omega)}+
h_1(z)\overline{h_1(\xi)}+ h_2(z)\overline{h_2(\xi)}+\cdots
\end{equation*}
for all $z, \xi\in \Omega$. Obviously, $V(\Omega)K_{\Omega}(0,0)=1$
and  $1\leq V(\Omega)K_{\Omega}(z,\bar{z})<+\infty$ for all $z\in
\Omega$. The generic norm of $\Omega$ is defined by
$$N_{\Omega}(z,\bar{\xi}):=\left(V(\Omega)K_{\Omega}(z,\bar{\xi})\right)^{-\frac{1}{p}}\;\;\; (z,\xi\in\Omega),$$
where $(V(\Omega)K_{\Omega}(z,\bar{\xi}))^{-\frac{1}{p}}:=\exp
(-\frac{1}{p}\ln (V(\Omega)K_{\Omega}(z,\bar{\xi})))$, in which
$\ln$ denotes the principal branch of logarithm (note
$K_{\Omega}(z,\bar{\xi})\neq 0$ for all $z,\xi\in \Omega$). Thus $
N_{\Omega}(0,0)=1$,  $0< N_{\Omega}(z,\bar{z})\leq 1$ for all $z\in
\Omega$ and $N_{\Omega}(z,\bar{z})=0$ on the boundary of $\Omega$.

Let $\Omega_i\subset \mathbb{C}^{d_i}$ be an irreducible bounded
symmetric domain $(1\leq i\leq k)$. For given positive integer
$d_0$,  positive real numbers $\mu_i$ $(1\leq i\leq k)$, the
generalized Cartan-Hartogs domain $ \big(\prod_{j=1}^k\Omega_j
\big)^{{\mathbb{B}}^{d_0}}(\mu)$ is defined by
\begin{equation}\label{eq1.1}
 \big(\prod_{j=1}^k\Omega_j\big)^{\mathbb{B}^{d_0}}(\mu):=\left\{(z,w)\in \prod_{j=1}^k \Omega_j\times \mathbb{B}^{d_0}
  : \|w\|^2<\prod_{j=1}^kN_{\Omega_j}(z_j,\overline{z_j})^{\mu_j} \right\},
\end{equation}
where $\mu=(\mu_1,\ldots,\mu_k) \in \mathbb{(R_+)}^k,\;
z=(z_1,\ldots,z_k)\in
\mathbb{C}^{d_1}\times\cdots\times\mathbb{C}^{d_k}$, $\|\cdot\|$ is
the standard Hermitian norm in $\mathbb{C}^{d_0}$,
$N_{\Omega_j}(z_j,\overline{z_j})$ is the generic norm of $\Omega_j$
$(1\leq i\leq k)$ and $\mathbb{B}^{d_0}:=\{w\in
\mathbb{C}^{d_0}:\|w\|^2<1\}$. Note
$\prod_{j=1}^kN_{\Omega_j}(0,0)^{\mu_j}= 1,$
$0<\prod_{j=1}^kN_{\Omega_j}(z_j,\overline{z_j})^{\mu_j}\leq 1$ on
$\prod_{j=1}^k\Omega_j$ and
$\prod_{j=1}^kN_{\Omega_j}(z_j,\overline{z_j})^{\mu_j}=0$ on the
boundary $\partial(\prod_{j=1}^k\Omega_j)$. Thus
$\partial(\prod_{j=1}^k\Omega_j)\subset
\partial\big(\big(\prod_{j=1}^k\Omega_j\big)^{\mathbb{B}^{d_0}}(\mu)\big).$
For the reference of the generalized Cartan-Hartogs domains, see
Ahn-Park \cite{AP}, Tu-Wang \cite{TW} and Wang-Hao \cite{WH}.

When $k\geq 2,$ then any generalized Cartan-Hartogs domain $
\big(\prod_{j=1}^k\Omega_j \big)^{{\mathbb{B}}^{d_0}}(\mu)$ is a
bounded nonhomogeneous domain in $\mathbb{C}^{d_0+\cdots+d_k}.$ In
fact, for $(z_1,\ldots,z_k,w)\in \prod_{j=1}^k \Omega_j\times
(\mathbb{C}^{d_0}\setminus \{0\}),$ we have
\begin{align*}
\ln \|w\|^2  - \ln
(\prod_{j=1}^kN_{\Omega_j}(z_j,\overline{z_j})^{\mu_j}) =\ln \|w\|^2
+ \sum\limits_{j=1}^m{\frac{\mu_j}{p_j}} \ln ( V(\Omega_j)
K_{\Omega_j} (z_j,\overline{z_j}))
\end{align*}
is real analytic on $\prod_{j=1}^k \Omega_j\times
(\mathbb{C}^{d_0}\setminus \{0\})$ and each ${\frac{\mu_j}{p_j}}\ln
( V(\Omega_j) K_{\Omega_j} (z_j,\overline{z_j}) )$ is a real
analytic strictly  plurisubharmonic function on $ \Omega_j$ $(1\leq
j \leq k)$. Thus, the generalized Cartan-Hartogs domain $
\big(\prod_{j=1}^k\Omega_j \big)^{{\mathbb{B}}^{d_0}}(\mu)$ is
strongly pseudoconvex at the boundary part
\begin{align*}
& \left\{(z,w)\in \prod_{j=1}^k \Omega_j\times \mathbb{C}^{d_0}
  : \ln \|w\|^2
+ \sum\limits_{j=1}^m{\frac{\mu_j}{p_j}} \ln ( V(\Omega_j)
K_{\Omega_j} (z_j,\overline{z_j}))=0,\;
  w\not=0 \right\}   \\
& \left( =\left\{(z,w)\in \prod_{j=1}^k \Omega_j\times
\mathbb{C}^{d_0}
  : \|w\|^2=\prod_{j=1}^kN_{\Omega_j}(z_j,\overline{z_j})^{\mu_j},\;
  w\not=0 \right\} \right). \end{align*}
Therefore, if a generalized Cartan-Hartogs domain $
\big(\prod_{j=1}^k\Omega_j \big)^{{\mathbb{B}}^{d_0}}(\mu)$ $(k\geq
2)$ is homogeneous, then the generalized Cartan-Hartogs domain must
be biholomorphic  to the unit ball by the Wong-Rosay theorem (see
Rudin \cite{R}, Theorem 15.5.10 and its Corollary), and thus, it is
a strongly pseudoconvex domain. From the boundary
$$b\big(\big(\prod_{j=1}^k\Omega_j\big)^{{\mathbb{B}}^{d_0}}(\mu)\big)\supset
b\big(\prod_{j=1}^k\Omega_j\big),$$  we have that the boundary $b
\big((\prod_{j=1}^k\Omega_j \big)^{{\mathbb{B}}^{d_0}}(\mu))$
$(k\geq 2)$ contains a positive-dimensional complex submanifold.
This is impossible, since, in our case, the generalized
Cartan-Hartogs domain is strongly pseudoconvex. So any generalized
Cartan-Hartogs domain $ \big(\prod_{j=1}^k\Omega_j
\big)^{{\mathbb{B}}^{d_0}}(\mu)$ $(k\geq 2)$ is a bounded
nonhomogeneous domain.

For the generalized Cartan-Hartogs domain $
\big(\prod_{j=1}^k\Omega_j \big)^{{\mathbb{B}}^{d_0}}(\mu)$, define
\begin{equation}\label{eq1.2}
 \Phi(z,w):=-\ln\left(\prod_{j=1}^kN_{\Omega_j}(z_j,\overline{z_j})^{\mu_j}-\|w\|^2\right).
\end{equation}
The K\"{a}hler form $\omega(\mu)$ on
$\big(\prod_{j=1}^k\Omega_j\big)^{{\mathbb{B}}^{d_0}}(\mu)$ is
defined by
\begin{equation}\label{eq1.3}
 \omega(\mu):=\frac{\sqrt{-1}}{2\pi}\partial
\overline{\partial}\Phi.
\end{equation}
 The canonical metric  $g(\mu)$ on  $\big(\prod_{j=1}^k\Omega_j\big)^{{\mathbb{B}}^{d_0}}(\mu)$
associated to  $\omega(\mu)$ is given by
$$ds^2=\sum_{i,j=1}^{n}\frac{\partial^2\Phi}{\partial Z_i \partial
\overline{Z_j}}dZ_i\otimes d\overline{Z_j},$$
where
$$n=\sum_{j=0}^kd_j,\;\;  Z=(Z_1,\ldots,Z_n):=(z,w).$$

Note that $\Phi(z,w)$ defined by \eqref{eq1.2} is also a real
analytic strictly plurisubharmonic exhaustion function for the
generalized Cartan-Hartogs domain
$\big(\prod_{j=1}^k\Omega_j\big)^{\mathbb{B}^{d_0}}(\mu).$ Thus $
\big(\prod_{j=1}^k\Omega_j \big)^{\mathbb{B}^{d_0}}(\mu)$ is a
bounded pseudoconvex domain in $\mathbb{C}^{d_0+\cdots+d_k}.$

For the generalized Cartan-Hartogs domain $ \big(
\big(\prod_{j=1}^k\Omega_j \big)^{{\mathbb{B}}^{d_0}}(\mu), g(\mu)
\big)$, we have (see Theorem \ref{Th:2.3} in this paper) that the
Rawnsley's $\varepsilon$-function admits the expansion:
\begin{equation}\label{eq1.5}
  \varepsilon_{\alpha}(z,w)=\sum_{j=0}^{n}a_j(z,w)\alpha^{n-j}, \;\; (z,w)\in
   \big(\prod_{j=1}^k\Omega_j \big)^{{\mathbb{B}}^{d_0}}(\mu),
\end{equation}
where $n=\sum_{j=0}^kd_j.$ By Th. 1.1 of Lu \cite{Lu},  Th. 4.1.2
and Th. 6.1.1 of Ma-Marinescu \cite{MM07}, Th. 3.11 of Ma-Marinescu
\cite{MM08} and Th. 0.1 of Ma-Marinescu \cite{MM12}, see also Th.
3.3 of Xu \cite{X}, we have
\begin{equation}\label{eq1.6}
\left\{  \begin{array}{ll}
    a_0 & =1, \\
    a_1 & = \frac{1}{2}k_g, \\
    a_2 & =\frac{1}{3}\triangle k_g+\frac{1}{24}|R|^2-\frac{1}{6}|Ric|^2+\frac{1}{8}k_g^2,
  \end{array}\right.
\end{equation}
where $k_g$, $\triangle$, $R$ and $Ric$  denote  the scalar
curvature, the Laplace, the curvature tensor and the Ricci curvature
associated to the canonical metric $g(\mu)$, respectively.

\vskip 5pt

In 2012, Loi-Zedda \cite{LZ} described balanced metrics on an
irreducible bounded symmetric domain $\Omega$ as follows.

\begin{Theorem}\textup{(Loi-Zedda \cite{LZ})} \label{Th:00}{Let $\Omega$ be an
irreducible bounded symmetric domain of genus $p$ equipped with its
Bergman metric $g_B$. Then the metric $\alpha g_B$ $(\alpha > 0)$ is
balanced if and only if $\alpha > \frac{p-1}{p}.$ }\end{Theorem}

Let $ {\mathbb{B}}^{d}$ be the unit ball in $\mathbb{C}^d$  and  let
the metric $g_{hyp}$ on  $ {\mathbb{B}}^{d}$  be given by $$
ds^2=-\sum_{i,j=1}^{d}\frac{\partial^2\ln(1-\|z\|^2)}{\partial z_i
\partial \overline{z_j}}dz_i\otimes d\overline{z_j}.$$
We call $(\mathbb{B}^d, g_{hyp})$ the complex hyperbolic space. Note
$g_{hyp}=\frac{1}{d+1} g_B$ on $ {\mathbb{B}}^{d}$ and the genus
$p=d+1$ for $ {\mathbb{B}}^{d}$. Thus, by Theorem \ref{Th:00}, we
have that $\alpha g_{hyp}\;(\alpha>0)$ is a balanced metric on $
{\mathbb{B}}^{d}$ if and only if $\alpha>d$.

In the special case of $k=1,$  the generalized Cartan-Hartogs domain
$\Omega^{{\mathbb{B}}^{d_0}}(\mu)$ is also called the Cartan-Hartogs
domain. When $\Omega=\mathbb{B}^d,\; \mu=1,$ we have
$\Omega^{{\mathbb{B}}^{d_0}}(\mu)= \mathbb{B}^{d+d_0}.$ In 2012,
Loi-Zedda \cite{LZ} gave a characterization of the complex
hyperbolic space among the Cartan-Hartogs domains in terms of
balanced metrics as follows.

\begin{Theorem}\textup{(Loi-Zedda \cite{LZ})} \label{Th:01} {
Let $(\Omega^{{\mathbb{B}}^{d_0}}(\mu), g(\mu))$ be a Cartan-Hartogs
domain over the irreducible bounded symmetric domain $\Omega$ in
$\mathbb{C}^d$ with the canonical metric $g(\mu)$. Then the metric
$\alpha g(\mu)$ $(\alpha > 0)$ on $\Omega^{{\mathbb{B}}^{d_0}}(\mu)$
is balanced if and only if $\alpha > d+d_0$ and
$(\Omega^{{\mathbb{B}}^{d_0}}(\mu), g(\mu))$ is biholomorphically
isometric to the complex hyperbolic space $({{\mathbb{B}}}^{d+d_0},
g_{hyp})$ $($i.e., $\Omega=\mathbb{B}^d,\; \mu=1).$}\end{Theorem}

In 2014, Feng-Tu \cite{FT} obtained the following improvement of
Theorems \ref{Th:01}.

\begin{Theorem}\textup{(Feng-Tu \cite{FT})}\label{Th:02}
{Let $(\Omega^{{\mathbb{B}}^{d_0}}(\mu), g(\mu))$ be a
Cartan-Hartogs domain over the irreducible bounded symmetric domain
$\Omega$ in $\mathbb{C}^d$ with the canonical  metric $g(\mu)$. Then
the  coefficient $a_2$ $($see \eqref{eq1.5}$)$ of the Rawnsley's
$\varepsilon$-function expansion is a constant on
$\Omega^{{\mathbb{B}}^{d_0}}(\mu)$ if and only if
$(\Omega^{{\mathbb{B}}^{d_0}}(\mu), g(\mu))$ is biholomorphically
isometric to the complex hyperbolic space $({{\mathbb{B}}}^{d+d_0},
g_{hyp})$ $($i.e., $\Omega=\mathbb{B}^d,\; \mu=1).$}\end{Theorem}

By Ligocka \cite{L2}, the Bergman kernel of the Hartogs type domain
$\big(\prod_{j=1}^k\Omega_j\big)^{{\mathbb{B}}^{d_0}}(\mu)$ can be
expressed as infinite sum in terms of the weighted Bergman kernels
of the base space $\prod_{j=1}^k\Omega_j$ with weights
$(\prod_{j=1}^kN_j(z_j,\overline{z_j})^{\mu_j})^m$ $(0\leq m<
+\infty)$. Ahn-Park \cite{AP} use the technique in Ligocka \cite{L2}
to obtain an explicit form of the Bergman kernel of the Hilbert
space of square integrable holomorphic functions on the generalized
Cartan-Hartogs domain
$\big(\prod_{j=1}^k\Omega_j\big)^{{\mathbb{B}}^{d_0}}(\mu)$ and
determine the condition that their Bergman kernel functions have
zeros in 2012.

By studying the explicit solutions of a class of complex
Monge-Amp\`ere equations on generalized Cartan-Hartogs domain
$\big(\prod_{j=1}^k\Omega_j\big)^{{\mathbb{B}}^{d_0}}(\mu)$,
Wang-Hao \cite{WH} described K\"ahler-Einstein metrics on such
domain in 2014.

Each generalized Cartan-Hartogs domain
$\big(\prod_{j=1}^k\Omega_j\big)^{{\mathbb{B}}^{d_0}}(\mu)$ $(k\geq
2)$ is a bounded nonhomogeneous domain. The purpose of this paper is
to obtain the existence of balanced metrics for such bounded
nonhomogeneous domains. In this paper, we obtain an explicit formula
for the Bergman kernel of the weighted Hilbert space
$\mathcal{H}_{\alpha}$ of square integrable holomorphic functions on
$\big(\big(\prod_{j=1}^k\Omega_j\big)^{{\mathbb{B}}^{d_0}}(\mu),
g(\mu)\big)$ with the weight $\exp\{-\alpha \Phi\}$ (where $\Phi$ is
a globally defined K\"{a}hler potential for the canonical metric
$g(\mu)$) for $\alpha>0$. Furthermore, we give an explicit
expression of the Rawnsley's $\varepsilon$-function expansion for
$\big(\big(\prod_{j=1}^k\Omega_j\big)^{{\mathbb{B}}^{d_0}}(\mu),
g(\mu)\big)$, and, use the asymptotics of the Rawnsley's
$\varepsilon$-function to draw necessary and sufficient conditions
for the metric  $\alpha g(\mu)$ on the generalized Cartan-Hartogs
domain $\big(\prod_{j=1}^k\Omega_j\big)^{{\mathbb{B}}^{d_0}}(\mu)$
to be balanced as follows.

\begin{Theorem} \label{Th:2} {Let $\Omega_i\subset \mathbb{C}^{d_i}$ be an irreducible bounded
symmetric domain, and denote  the rank $r_i$, the characteristic
multiplicities $a_i, b_i$, the dimension $d_i$ and the genus $p_i$
for $\Omega_i$ $(1\leq i \leq k)$. For given positive integer $d_0$,
positive real numbers $\mu_i$ $(1\leq i\leq k)$, let $g(\mu)$ be the
canonical  metric on the generalized Cartan-Hartogs domain
$\big(\prod_{j=1}^k\Omega_j\big)^{{\mathbb{B}}^{d_0}}(\mu)$. Then we
have the results as follows.

$(i)$ The metric  $\alpha g(\mu)$ on the generalized Cartan-Hartogs
domain $\big(\prod_{j=1}^k\Omega_j\big)^{{\mathbb{B}}^{d_0}}(\mu)$
is balanced if and only if
 $$\alpha>\max\left\{\sum_{j=0}^kd_j,\frac{p_1-1}{\mu_1},\ldots,\frac{p_k-1}{\mu_k}\right\}$$
 and
\begin{equation}\label{eq1.7}
   \prod_{i=1}^k \prod_{j=1}^{r_i}\left(\mu_i
   x-p_i+1+(j-1)\frac{a_i}{2}\right)_{1+b_i+(r_i-j)a_i}=\prod_{j=1}^k\mu_j^{d_j}\cdot\prod_{j=1}^{d}(x-j),
\end{equation}
where $d=\sum_{j=1}^kd_j$ and
$(x)_m=\frac{\Gamma(x+m)}{\Gamma(x)}=x(x+1)(x+2)\cdots(x+m-1)$.

$(ii)$ If the metric  $\alpha g(\mu)$ on the domain
$\big(\prod_{j=1}^k\Omega_j\big)^{{\mathbb{B}}^{d_0}}(\mu)$  is
balanced, then each $\Omega_j$ $(1\leq j\leq k)$ must be
biholomorphic to
\begin{equation*}
 \Omega_{I}(1,n)\equiv {\mathbb{B}}^n:=\left\{z\in
 \mathbb{C}^{n}:\|z\|^2<1\right\},\;\;
 \Omega_{III}(2):=\left\{z\in \mathcal{M}_{2,2}: z^t=z, I-z{\overline z}^t>0\right\} ,
\end{equation*}
or
\begin{equation*}
  \Omega_{IV}(m):=\left\{z\in \mathbb{C}^{m}: 1-2z{\overline z}^t+|zz^t|^2>0,\|z\|^2<1\right\}~~(m\geq 5~\text{and}~m~\text{are odd}).
\end{equation*}

$(iii)$ For
$\alpha>\max\left\{d_0+d_1+d_2,\frac{p_1-1}{\mu_1},\frac{p_2-1}{\mu_2}\right\}$,
then the metric  $\alpha g(\mu)$ on the domain $\left(\Omega_1
\times \Omega_2 \right)^{{\mathbb{B}}^{d_0}}(\mu)$ is balanced if
and only if
 $(\left(\Omega_1 \times  \Omega_2 \right)^{{\mathbb{B}}^{d_0}}(\mu), \alpha g(\mu))$
 is biholomorphically isometric to
\begin{eqnarray*}
\left( \big({\mathbb{B}}^d\times
 {\mathbb{B}}\big)^{{\mathbb{B}}^{d_0}}\big(1,\frac{1}{d+1}\big),\;  g(1,\frac{1}{d+1})\right)   \;\;
 \mbox{or}   \;\;
\left(\left({\mathbb{B}}\times\Omega_{III}(2)\right)^{{\mathbb{B}}^{d_0}}\big(1,\frac{1}{2}\big),\;
g(1,\frac{1}{2})\right).
\end{eqnarray*}
}\end{Theorem}

\vskip 5pt
 For example, for $\Omega_{IV}(5)\times\prod_{j=1}^{3}{\mathbb{B}}$ and
 $\mu=\left(\frac{1}{2},1,\frac{1}{3},\frac{1}{7}\right)$,
 since
\begin{eqnarray*}
   & &\prod_{i=1}^4 \prod_{j=1}^{r_i}\left(\mu_i
   x-p_i+1+(j-1)\frac{a_i}{2}\right)_{1+b_i+(r_i-j)a_i}  \\
   &=&\prod_{j=1}^4\left(\mu_1x-j\right)\cdot\big(\mu_1x-\frac{5}{2}\big)\prod_{j=2}^4(\mu_jx-1)  \\
   &=&\frac{1}{2^5\times 3\times
  7}\prod_{j=1}^8(x-j),
\end{eqnarray*}
by Theorem \ref{Th:2}(i), when $\alpha>8+d_0$, the metric  $\alpha
g(\mu)$ on the generalized Cartan-Hartogs domain
$$\big(\Omega_{IV}(5)\times\prod_{j=1}^{3}{\mathbb{B}}\big)^{{\mathbb{B}}^{d_0}}\big(\frac{1}{2},1,\frac{1}{3},\frac{1}{7}\big)$$
is balanced, where $g(\mu)$ is the canonical  metric. So the example
complements Theorem \ref{Th:2}(ii).

As a supplement to our main result, we will prove the following
result.

\begin{Theorem}\label{Th:1}{
Let $\Omega_i\subset \mathbb{C}^{d_i}$ be an irreducible bounded
symmetric domain $(i=1,2)$. For any positive integer $d_0$, then
there exist positive numbers $\mu_1, \mu_2$ such that the
coefficient $a_2$ $($see \eqref{eq1.5}$)$ of the Rawnsley's
$\varepsilon$-function expansion of
$$\left((\Omega_1\times\Omega_2)^{{\mathbb{B}}^{d_0}}(\mu_1,\mu_2),
g(\mu_1,\mu_2)\right)$$ is a constant. }\end{Theorem}

\begin{Remark} By Theorem \ref{Th:2}(iii) and Theorem \ref{Th:1}, there
exists generalized Cartan-Hartogs domain
$$\left((\Omega_1\times\Omega_2)^{{\mathbb{B}}^{d_0}}(\mu_1,\mu_2),
g(\mu_1,\mu_2)\right)$$ such that the coefficient $a_2$ of its
Rawnsley's $\varepsilon$-function expansion is constant, but the
metric  $\alpha g(\mu_1,\mu_2)$ on
$$(\Omega_1\times\Omega_2)^{{\mathbb{B}}^{d_0}}(\mu_1,\mu_2)$$ is not
balanced for  all $\alpha>0$  (cf. Theorem \ref{Th:02}).
\end{Remark}

For $k\geq 2$, each generalized Cartan-Hartogs domain
$(\big(\prod_{j=1}^k\Omega_j\big)^{{\mathbb{B}}^{d_0}}(\mu), g(\mu)
\big)$ is a noncompact, complete (the argument is analogous to
\cite{YW}) and nonhomogeneous K\"{a}hler manifold. By studying the
explicit solutions of a class of complex Monge-Amp\`ere equations on
generalized Cartan-Hartogs domains, Wang-Hao \cite{WH} described
K\"ahler-Einstein metrics on such domains in 2014. From the proof in
Wang-Hao \cite{WH} and \eqref{eq2.6} in this paper, one can state
the result as follows.

\begin{Theorem}\label{Th:4}\textup{(Wang-Hao \cite{WH})} {Let $\Omega_i\subset \mathbb{C}^{d_i}$ be an irreducible bounded
symmetric domain and let $p_i$ and $d_i$ be the genus and the
dimension respectively for $\Omega_i$, $1\leq i \leq k$. Then
$\mu_i=\frac{p_i}{\sum_{j=1}^k d_j+1}$ $(1\leq i \leq k)$ if and
only if the canonical metric $g(\mu)$  on the domain
$\big(\prod_{j=1}^k\Omega_j\big)^{{\mathbb{B}}^{d_0}}(\mu)$ is
K\"{a}hler-Einstein.
 }\end{Theorem}

\begin{Remark}
When $\alpha>8+d_0$, the metric  $\alpha
g(\frac{1}{2},1,\frac{1}{3},\frac{1}{7})$ on the generalized
Cartan-Hartogs domain
$$\big(\Omega_{IV}(5)\times\prod_{j=1}^{3}{\mathbb{B}}\big)^{{\mathbb{B}}^{d_0}}\big(\frac{1}{2},1,\frac{1}{3},\frac{1}{7}\big)$$
is balanced by Theorem \ref{Th:2}(i), but is not K\"{a}hler-Einstein
for  all $\alpha>0$ by Theorem \ref{Th:4}, where
$g(\frac{1}{2},1,\frac{1}{3},\frac{1}{7})$ is the canonical  metric.
\end{Remark}

\begin{Remark} The metric $\alpha g(\frac{2}{3},\frac{2}{3})$  $(\alpha>0)$  on the generalized Cartan-Hartogs
domain
\begin{equation*}
   \left(\mathbb{B}\times\mathbb{B}\right)^{\mathbb{B}}(\frac{2}{3},\frac{2}{3})
\end{equation*}
is K\"{a}hler-Einstein by Theorem \ref{Th:4}, but is not balanced
for all $\alpha>0$ by Theorem \ref{Th:2}(iii), where $
g(\frac{2}{3},\frac{2}{3})$ is the canonical metric.
\end{Remark}

The paper is organized as follows. In Section 2, we obtain an
explicit formula for the Bergman kernel $K_{\alpha}$ of the weighted
Hilbert space $\mathcal{H}_{\alpha}$ of square integrable
holomorphic functions on  a generalized Cartan-Hartogs domain
$\big(\big(\prod_{j=1}^k\Omega_j\big)^{{\mathbb{B}}^{d_0}}(\mu),
g(\mu)\big)$ with the weight  $\exp\{-\alpha \Phi\}$, in terms of
ranks, Hua polynomials and generic norms of ${\mathbb{B}}^{d_0}$ and
$\Omega_i$ $(1\leq i \leq k)$. In Section 3, using results in
Section 2, we give the explicit expansion of the Rawnsley's
$\varepsilon$-function of
$\big(\big(\prod_{j=1}^k\Omega_j\big)^{{\mathbb{B}}^{d_0}}(\mu),
g(\mu)\big)$. In Section 4, using the explicit expression of the
Rawnsley's $\varepsilon$-function expansion, we obtain the necessary
and sufficient conditions (Theorem \ref{Th:2} in the paper) for the
metric $\alpha g(\mu)$ on the generalized Cartan-Hartogs domain
$\big(\prod_{j=1}^k\Omega_j\big)^{{\mathbb{B}}^{d_0}}(\mu)$ to be a
balanced metric. In Section 5, we will give the proof of Theorem
\ref{Th:1}.

 \setcounter{equation}{0}
\section{The reproducing kernel of $\mathcal{H}_{\alpha}$ for
$\big(\prod_{j=1}^k\Omega_j\big)^{{\mathbb{B}}^{d_0}}(\mu)$  with
the canonical metric $g(\mu)$ }

Let $\Omega$ be an irreducible bounded symmetric domain  in
$\mathbb{C}^d$ in its Harish-Chandra realization. Thus $\Omega$ is
the open unit ball of a Banach space which admits the structure of a
$JB^{\ast}$-triple. We denote the rank $r$, the characteristic
multiplicities $a,b$, the dimension $d$,  the genus $p$, and the
generic norm $N_{\Omega}(z,\overline{w})$ for $\Omega$. Thus
\begin{equation}\label{1.1}
    d=\frac{r(r-1)}{2}  a+rb+r,\quad   p=(r-1)a+b+2.
\end{equation}
For any $s>-1$, the value of the Hua integral
$\int_{\Omega}N_{\Omega}(z,\overline{z})^s dm(z)$ is given by
\begin{equation}\label{1.2}
\int_{\Omega}N_{\Omega}(z,\overline{z})^s
dm(z)=\frac{\chi(0)}{\chi(s)}\int_{\Omega}dm(z),
\end{equation}
where $dm(z)$ denotes the Euclidean measure on $\mathbb{C}^d$, $\chi$
is the Hua polynomial
\begin{equation}\label{1.3}
   \chi(s):=\prod_{j=1}^r\left(s+1+(j-1)\frac{a}{2}\right
   )_{1+b+(r-j)a},
\end{equation}
in which, for a  non-negative integer $m$, $(s)_m$ denotes the
raising factorial
$$ {(s)_m:=\frac{\Gamma(s+m)}{\Gamma(s)}=s(s+1)\cdots (s+m-1)}.$$

Let $\mathcal{G}$ stand for the identity connected component of the
group of biholomorphic self-maps of $\Omega$, and $\mathcal{K}$ {for
the stabilizer} of the origin in $\mathcal{G}$. Under the action
$f\mapsto f\circ k \; (k\in \mathcal{K})$ of $\mathcal{K}$, the
space $\mathcal{P}$ of holomorphic polynomials on $\mathbb{C}^d$
admits the Peter-Weyl decomposition
$$\mathcal{P}=\bigoplus_{\lambda}\mathcal{P}_{\lambda},$$
 {where the summation is taken over all partitions}
$\lambda$, i.e., $r$-tuples $(\lambda_1, \lambda_2, \cdots,
\lambda_r)$ of nonnegative integers such that $\lambda_1\geq
\lambda_2 \geq \cdots \geq \lambda_r \geq 0$,  {and the spaces}
$\mathcal{P}_{\lambda}$ are $\mathcal{K}$-invariant and irreducible.
For each $\lambda$, $\mathcal{P}_{\lambda} \subset
\mathcal{P}_{|\lambda|}$, where $|\lambda|$ denotes the weight of
partition $\lambda$, i.e., $|\lambda|:=\sum_{j=1}^r \lambda_j$,  and
$\mathcal{P}_{|\lambda|}$ is the space of homogeneous holomorphic
polynomials of degree $|\lambda|$.

Let
\begin{equation}\label{1.4}
    {\langle}f,g {\rangle}_{\mathcal{F}}:=\int_{\mathbb{C}^d}f(z)\overline{g(z)} d\rho_{\mathcal{F}}(z)
\end{equation}
be the Fock-Fischer inner product on the space $\mathcal{P}$ of
holomorphic polynomials on $\mathbb{C}^d$, where
\begin{equation}\label{1.5}
    d\rho_{\mathcal{F}}(z):=\exp\{-m(z,\overline{z})\}
    \frac{\left(\frac{\sqrt{-1}}{2\pi}\partial\overline{\partial}m(z,\overline{z})\right)^d}{d!}
\end{equation}
and $m(z,\overline{z}):=-\left.\frac{\partial \ln
N_{\Omega}(tz,\overline{z})}{\partial
t}\right|_{t=0}=-\left.\frac{\partial
N_{\Omega}(tz,\overline{z})}{\partial t}\right|_{t=0}.$

 For every partition $\lambda$, let $K_{\lambda}(z_1,\overline{z_2})$ be
the reproducing kernel of $\mathcal{P}_{\lambda}$ with respect to
\eqref{1.4}. The weighted Bergman kernel of the weighted Hilbert
space $A^2(\mathbb{C}^d,\rho_{\mathcal{F}})$ of square-integrable
holomorphic functions on $\mathbb{C}^d$ with the measure
$d\rho_{\mathcal{F}}$ is
\begin{equation}\label{1.6}
    K(z_1,\overline{z_2}):=\sum_{\lambda}K_{\lambda}(z_1,\overline{z_2}).
\end{equation}

The kernels $K_{\lambda}(z_1,\overline{z_2})$ are related to the
generic norm  $N_{\Omega}(z_1,\overline{z_2})$ by the
Faraut-{Kor\'{a}nyi} formula
\begin{equation}\label{1.7}
N_{\Omega}(z_1,\overline{z_2})^{-s}=\sum_{\lambda}(s)_{\lambda}K_{\lambda}(z_1,\overline{z_2}),
\end{equation}
where the series converges  {uniformly} on compact subsets of
$\Omega\times\Omega$, $s\in \mathbb{C}$,  in which  $(s)_{\lambda}$
denote the generalized Pochhammer symbol
\begin{equation}\label{1.8}
   (s)_{\lambda}:=\prod_{j=1}^r\big(s-\frac{j-1}{2}a\big)_{\lambda_j}.
\end{equation}
For the proofs of above facts and additional details, we refer,
e.g.,  to \cite{FK}, \cite{FKKLR} and \cite{YLR}.

\begin{Lemma}\label{Le:2.1}{
Let $\Omega_i$ be  an irreducible bounded symmetric domain  in
$\mathbb{C}^{d_i}$ in its Harish-Chandra realization, and denote the
generic norm $N_{\Omega_i}$ and the genus $p_i$ for $\Omega_i$
$(1\leq i \leq k)$. For $z_i^0\in \Omega_i$, let $\phi_i$ be an
automorphism of $\Omega_i$ such that $\phi_i(z^0_i)=0$, $1\leq i\leq
k$. By \cite{WH}, the function
\begin{equation}\label{eq2.1}
   \psi(z_1,\ldots,z_k):=\prod_{i=1}^k\frac{N_{\Omega_i}(z_i^0,\overline{z_i^0})^{\frac{\mu_i}{2}}}{N_{\Omega_i}(z_i,\overline{z_i^0})^{\mu_i}}
\end{equation}
satisfies
\begin{equation}\label{eq2.2}
    |\psi(z_1,\ldots,z_k)|^2=\prod_{i=1}^k\Big(\frac{N_{\Omega_i}(\phi_i(z_i),\overline{\phi_i(z_i)})}{N_{\Omega_i}(z_i,\overline{z_i})}\Big)^{\mu_i}.
\end{equation}
Define the mapping
\begin{equation}\label{eq2.3}
    \begin{array}{rcl}
   F: \big(\prod_{j=1}^k\Omega_j\big)^{{\mathbb{B}}^{d_0}}(\mu_1,\ldots,\mu_k) & \longrightarrow   & \big(\prod_{j=1}^k\Omega_j\big)^{{\mathbb{B}}^{d_0}}(\mu_1,\ldots,\mu_k), \\
     (z_1,\ldots,z_k,w)            & \longmapsto   & (\phi_1(z_1),\ldots,\phi_k(z_k),\psi(z_1,\ldots,z_k)w).
  \end{array}
\end{equation}
Then $F$ is an isometric automorphism of
$\big(\big(\prod_{j=1}^k\Omega_j\big)^{{\mathbb{B}}^{d_0}}(\mu_1,\ldots,\mu_k),g(\mu_1,\ldots,\mu_k)\big)$,
that is
\begin{equation}\label{eq2.4}
   \partial\overline{\partial}(\Phi(F(z_1,\ldots,z_k,w)))=\partial\overline{\partial}(\Phi(z_1,\ldots,z_k,w)),
\end{equation}
where
$\Phi(z_1,\ldots,z_k,w):=-\ln\left(\prod_{i=1}^kN_{\Omega_i}(z_i,\overline{z_i})^{\mu_i}-\|w\|^2\right)$
(see \eqref{eq1.2}).
 }\end{Lemma}

\begin{proof}[Proof]
It is easy to see that $F$ is an automorphism of
$\big(\prod_{j=1}^k\Omega_j\big)^{{\mathbb{B}}^{d_0}}(\mu_1,\ldots,\mu_k)$,
and
\begin{equation}\label{eq2.5}
    N_{\Omega_i}(\phi_i(z_i),\overline{\phi_i(z_i)})^{p_i}=J\phi_i(z_i)N_{\Omega_i}(z_i,\overline{z_i})^{p_i}\overline{J\phi_i(z_i)},
\end{equation}
where $J\phi_i(z_i)$ is the holomorphic Jacobian of the automorphism
$\phi_i$ of $\Omega_i$, $1\leq i\leq k$.

By  \eqref{eq2.2} and \eqref{eq2.5}, we have
\begin{eqnarray*}
   & & \prod_{i=1}^kN_{\Omega_i}(\phi_i(z_i),\overline{\phi_i(z_i)})^{\mu_i} -\|\psi(z_1,\ldots,z_k)w\|^2\\
   &=& \prod_{i=1}^k N_{\Omega_i}(\phi_i(z_i),\overline{\phi_i(z_i)})^{\mu_i}\left(1-\frac{\|w\|^2}{\prod_{i=1}^kN_{\Omega_i}(z_i,\overline{z_i})^{\mu_i}}\right) \\
   &=&\prod_{i=1}^k|J\phi_i(z_i)|^{\frac{2\mu_i}{p_i}}\left(\prod_{i=1}^kN_{\Omega_i}(z_i,\overline{z_i})^{\mu_i}-\|w\|^2\right),
\end{eqnarray*}
which implies \eqref{eq2.4}.
\end{proof}

\begin{Lemma}\label{Le:2.2}{
Let $\Omega_i$ be  an irreducible bounded symmetric domain  in
$\mathbb{C}^{d_i}$ in its Harish-Chandra realization, and denote the
generic norm $N_{\Omega_i}(z_i,\overline{z_i})$, the dimension $d_i$
and the genus $p_i$ for $\Omega_i$ $(1\leq i \leq k)$. Then we have
\begin{equation}\label{eq2.6}
   (\partial\overline{\partial}\Phi)^n=\frac{\prod_{i=1}^k\left(\mu_i^{d_i}C_{\Omega_i}N_{\Omega_i}(z_i,\overline{z_i})^{\mu_i(\sum_{j=1}^kd_j+1)-p_i}\right)}{\left(\prod_{i=1}^kN_{\Omega_i}(z_i,\overline{z_i})^{\mu_i}-\|w\|^2\right)^{n+1}}
   \left(\sum_{j=1}^n dZ_j\wedge d\overline{Z_j}\right)^n,
\end{equation}
where
 $$\Phi(z_1,\ldots,z_k,w)=-\ln\left(\prod_{i=1}^kN_{\Omega_i}(z_i,\overline{z_i})^{\mu_i}-\|w\|^2\right),$$
 $$ C_{\Omega_i}=\left.\det\left(-\frac{\partial^2\ln N_{\Omega_i}(z_i,\overline{z_i})}{\partial z_i^t\partial{\overline{z_i}}}\right)\right|_{z_i=0},$$
$$n=\sum_{j=0}^kd_j,\;\;  Z=(Z_1,\ldots,Z_n)=(z_1,\ldots,z_k,w).$$
 }\end{Lemma}

\begin{proof}[Proof]
It is well known that
\begin{equation}\label{eq2.7}
    \frac{(\frac{\sqrt{-1}}{2\pi}\partial\overline{\partial}\Phi)^n}{n!}=\det\left(\frac{\partial^2\Phi}{\partial Z^t\partial
   \overline{Z}}\right)\frac{\omega_0^n}{n!},
\end{equation}
where $\omega_0=\frac{\sqrt{-1}}{2\pi}\sum_{j=1}^ndZ_j\wedge
d\overline{Z_j}$, $\frac{\partial}{\partial
Z^t}=(\frac{\partial}{\partial Z_1},\frac{\partial}{\partial
Z_2},\ldots,\frac{\partial}{\partial Z_n})^t$,
$\frac{\partial}{\partial \overline{Z}}=(\frac{\partial}{\partial
\overline{Z_1}},\frac{\partial}{\partial
\overline{Z_2}},\ldots,\frac{\partial}{\partial \overline{Z_n}})$
and $\frac{\partial^2}{\partial Z^t\partial
\overline{Z}}=\frac{\partial}{\partial Z^t}\frac{\partial}{\partial
\overline{Z}}$.

From \eqref{eq2.4} and \eqref{eq2.7}, we get
\begin{equation}\label{eq2.8}
\det\left(\frac{\partial^2\Phi(F)}{\partial Z^t\partial
   \overline{Z}}\right)=\det\left(\frac{\partial^2\Phi}{\partial Z^t\partial
   \overline{Z}}\right).
\end{equation}
By the identity
\begin{equation}\label{eq2.9}
\frac{\partial^2\Phi(F)}{\partial Z^t\partial
   \overline{Z}}=\frac{\partial F}{\partial Z^t}\frac{\partial^2\Phi}{\partial Z^t\partial
   \overline{Z}}(F(Z))\overline{\left(\frac{\partial F}{\partial Z^t}\right)}^t
\end{equation}
 and \eqref{eq2.8}, we deduce
\begin{equation}\label{eq2.10}
\det\left(\frac{\partial^2\Phi}{\partial Z^t\partial
   \overline{Z}}\right)(Z)=|JF(Z)|^2\det\left(\frac{\partial^2\Phi}{\partial Z^t\partial
   \overline{Z}}\right)(F(Z)),
\end{equation}
where
\begin{equation*}
F:=(F_1,F_2,\ldots,F_n),~~\frac{\partial F}{\partial
Z^t}:=(\frac{\partial F_1}{\partial
   Z^t},\frac{\partial F_2}{\partial
   Z^t},\ldots,\frac{\partial F_n}{\partial
   Z^t})
\end{equation*}
and
\begin{equation*}\label{eq2.11}
   JF(Z):=\det\left(\frac{\partial F}{\partial
   Z^t}\right)(Z).
\end{equation*}

Let $Z^0=(z_1^0,\ldots,z_k^0,w^0)\in
\left(\prod_{j=1}^k\Omega_j\right)^{{\mathbb{B}}^{d_0}}(\mu_1,\ldots,\mu_k)$,
$\widetilde{Z^0}:=(\widetilde{z_1^0},\ldots,\widetilde{z_k^0},\widetilde{w^0})=F(Z^0)$.
By \eqref{eq2.1} and \eqref{eq2.3}, then
$$\widetilde{Z^0}=\left(0,\ldots,0,\frac{w^0}{\prod_{i=1}^kN_{\Omega_i}(z_i^0,\overline{z_i^0})^{\frac{\mu_i}{2}}}\right)$$
and
\begin{equation}\label{eq2.12}
 |JF(Z^0)|^2=\prod_{i=1}^k|J\phi_i(z_i^0)|^2\cdot|\psi(z_1^0,\ldots,z_k^0)|^{2d_0}.
\end{equation}
Using $N_{\Omega_i}(0,z_i)=1$, \eqref{eq2.1},  \eqref{eq2.5},
\eqref{eq2.12} and \eqref{eq2.10}, we have
\begin{equation}\label{eq2.13}
    |JF(Z^0)|^2=\prod_{i=1}^k\frac{1}{N_{\Omega_i}(z_i^0,\overline{z_i^0})^{p_i+\mu_i d_0}},
\end{equation}
and
\begin{equation}\label{eq2.13}
\det\left(\frac{\partial^2\Phi}{\partial Z^t\partial
   \overline{Z}}\right)(Z^0)=\prod_{i=1}^k\frac{1}{N_{\Omega_i}(z_i^0,\overline{z_i^0})^{p_i+\mu_i d_0}}\det\left(\frac{\partial^2\Phi}{\partial Z^t\partial
   \overline{Z}}\right)(\widetilde{Z^0}).
\end{equation}

A direct calculation gives
 \begin{equation}\label{eq2.24}
  \frac{\partial^2\Phi}{\partial Z^t\partial\overline{Z}}(0,\ldots,0,w)=\left(
                                                            \begin{array}{cccc}
                                                              \frac{\mu_1}{1-\|w\|^2}C_{d_1}  &\cdots & 0      & 0 \\
                                                              \vdots                          &\cdots & \vdots & \vdots \\
                                                              0                               &\cdots &  \frac{\mu_k}{1-\|w\|^2}C_{d_k} & 0 \\
                                                              0  &\cdots & 0 & \frac{1}{1-\|w\|^2}I_{d_0}+\frac{1}{(1-\|w\|^2)^2}\overline{w}^{t}w \\
                                                            \end{array}
                                                          \right),
 \end{equation}
where $I_{d_0}$ denotes  the $d_0\times d_0$ identity matrix,
$\overline{w}^{t}$ is  the complex conjugate transpose of the row
vector $w=(w_1,w_2,\cdots,w_{d_0})$, and
$C_{d_i}=-\left.\frac{\partial^2\ln N_{\Omega_i}}{\partial
z_i^t\partial\overline{z_i}}\right|_{z_i=0}$.

From \eqref{eq2.24}, we have
\begin{equation}\label{eq2.25}
    \det\left(\frac{\partial^2\Phi}{\partial Z^t\partial
   \overline{Z}}\right)(0,\ldots,0,w)=\frac{\prod_{i=1}^k(\mu_i^{d_i}\det C_{d_i}) }{(1-\|w\|^2)^{\sum_{j=0}^kd_j+1}}.
\end{equation}
Finally, combining \eqref{eq2.13} and \eqref{eq2.25}, we have
\eqref{eq2.6}.
\end{proof}

\begin{Theorem}\label{Th:a.2}{Let $\Omega_i$ be an irreducible bounded symmetric domain  in
$\mathbb{C}^{d_i}$ in its Harish-Chandra realization, and denote the
generic norm $N_{\Omega_i}$,  the genus $p_i$, the dimension $d_i$
and the Hua polynomial $\chi_i$ (see \eqref{1.3}) for $\Omega_i$
$(1\leq i \leq k)$. Let the generalized Cartan-Hartogs domain
$\big(\prod_{j=1}^k\Omega_j\big)^{{\mathbb{B}}^{d_0}}(\mu)$ be
endowed with the canonical metric $g(\mu)$. Set $n=\sum_{j=0}^kd_j$.
For
$\alpha>\max\{n,\frac{p_1-1}{\mu_1},\ldots,\frac{p_k-1}{\mu_k}\}$,
then the Bergman kernel $K_{\alpha}(Z;\overline{Z})$ of the weighted
 Hilbert space
$$\mathcal{H}_{\alpha}=\left\{ f\in \mbox{\rm Hol}\big(\big(\prod_{j=1}^k\Omega_j\big)^{{\mathbb{B}}^{d_0}}(\mu)\big):
 \int_{\left(\prod_{j=1}^k\Omega_j\right)^{{\mathbb{B}}^{d_0}}(\mu)}|f|^2\exp\{-\alpha \Phi\}
\frac{\omega(\mu)^{n}}{n!}<+\infty\right\}$$ can be written as
\begin{eqnarray}
\nonumber      && K_{\alpha}(Z;\overline{Z})  \\
\label{eq2.26}
               &=&\left.\frac{(\alpha-n)_{d_0}}{\prod_{i=1}^k\mu_i^{d_i}}\prod_{i=1}^k\left(\frac{1}{N_{\Omega_i}(z_i,\overline{z_i})}\right)^{\mu_i\alpha}
                       \prod_{i=1}^k\chi_i\big(\mu_i(\alpha+t\frac{d}{dt})-p_i\big)\frac{1}{\left(1-t\frac{\|w\|^2}
                       {\prod_{i=1}^kN_{\Omega_i}(z_i,\overline{z_i})^{\mu_i}}\right)^{\alpha-d}}\right|_{t=1},
\end{eqnarray}
where $Z=(z_1,\ldots,z_k,w)$,\; $d=\sum_{j=1}^kd_j$. }\end{Theorem}

\begin{proof}[Proof] By  \eqref{eq2.6}, the inner product on
$\mathcal{H}_{\alpha}$ is given by
\begin{eqnarray}
\nonumber (f,g)  &=& \frac{\prod_{i=1}^k(\mu_i^{d_i}C_{\Omega_i})}{\pi^{n}}\int_{\left(\prod_{j=1}^k\Omega_j\right)^{{\mathbb{B}}^{d_0}}(\mu)}f(Z)\overline{g(Z)} \\
\nonumber   & &
\times\prod_{i=1}^kN_{\Omega_i}(z_i,\overline{z_i})^{\mu_i(\alpha-d_0)-p_i}\Big(1-\frac{\|w\|^2}
{\prod_{i=1}^kN_{\Omega_i}(z_i,\overline{z_i})^{\mu_i}}\Big)^{\alpha-n-1}dm(Z),
\end{eqnarray}
where $dm$ denotes the Euclidean measure.

For convenience, we set $\Omega_0={\mathbb{B}}^{d_0}$, $z_0=w$.
Denote the rank $r_i$, the characteristic multiplicities $a_i, b_i$,
the dimension $d_i$, the genus $p_i$, the Hua polynomial $ \chi_i$,
the generalized Pochhammer symbol $(s)_\mathbf{\lambda}^{(i)}$, the
generic norm $N_{\Omega_i}$ and the Euclidean volume $V(\Omega_i)$
for the irreducible bounded symmetric domain $\Omega_i$, $0\leq
i\leq k$.

Let $\mathcal{G}_i$ stand for the identity connected components of
 groups of biholomorphic self-maps of $\Omega_i \subset \mathbb{C}^{d_i}$, and $\mathcal{K}_i$
for stabilizers of the origin in $\mathcal{G}_i$, respectively. For
any $u=(u_0,\ldots,u_k)\in \mathcal{K}:= \mathcal{K}_0\times\ldots\times\mathcal{K}_k$,
we define the action
 $$ \pi(u)f(z_1,\ldots,z_k,w)\equiv f\circ u(z_1,\ldots,z_k,w):=f(u_1\circ z_1,\ldots,u_k\circ z_k,u_0\circ w)$$
 of $\mathcal{K}$, then the space $\mathcal{P}$
of holomorphic polynomials on $\prod_{j=0}^k\mathbb{C}^{d_j}$ admits the Peter-Weyl decomposition
$$\mathcal{P}=\bigoplus_{{\ell(\lambda_i)\leq r_i\atop 0\leq i\leq k}}\mathcal{P}^{(0)}_{\lambda_0}\otimes \ldots\otimes \mathcal{P}^{(k)}_{\lambda_k},$$
where  the space $\mathcal{P}^{(i)}_{\lambda_i}$ is
$\mathcal{K}_i$-invariant and irreducible subspace of the space
 of holomorphic polynomials on $\mathbb{C}^{d_i}$, and $\ell(\lambda_i)$ denotes the length of partition $\lambda_i$ $(0\leq i \leq k)$.

Since $\mathcal{H}_{\alpha}$ is invariant under the action of
$\mathcal{K}$, namely, $\forall u\in
\mathcal{K}$, $(\pi(u)f,\pi(u)g)=(f,g)$,
 $\mathcal{H}_{\alpha}$ admits
an irreducible decomposition (see \cite{Far-Tho})
$$\mathcal{H}_{\alpha}=\widehat{\bigoplus_{{\ell(\lambda_i)\leq r_i\atop 0\leq i\leq k}}}\mathcal{P}^{(0)}_{\lambda_0}\otimes \ldots\otimes \mathcal{P}^{(k)}_{\lambda_k},$$
where $\widehat{\bigoplus}$ denotes the orthogonal direct sum.

For given partition $\lambda_i$ of length $\leq r_i$, let
$K^{(i)}_{\lambda_i}(z_i;\overline{z_i})$ be the reproducing kernel
of $\mathcal{P}^{(i)}_{\lambda_i}$ with respect to \eqref{1.4}. By
Schur's lemma, there exists a positive constant
$c_{\lambda_0\ldots\lambda_k}$ such that
$c_{\lambda_0\ldots\lambda_k}\prod_{j=0}^kK^{(j)}_{\lambda_j}(z_j;\overline{z_j})$
is the reproducing kernel of $\mathcal{P}^{(0)}_{\lambda_0}\otimes
\ldots\otimes \mathcal{P}^{(k)}_{\lambda_k}$ with respect to the
above inner product $(\cdot,\cdot)$. From the definition of the
reproducing kernel, we have
\begin{eqnarray*}
   & & \frac{\prod_{i=1}^k(\mu_i^{d_i}C_{\Omega_i})}{\pi^{n}}\int_{\left(\prod_{j=1}^k\Omega_j\right)^{{\mathbb{B}}^{d_0}}(\mu)}c_{\lambda_0\ldots\lambda_k}\prod_{j=0}^kK^{(j)}_{\lambda_j}(z_j;\overline{z_j}) \\
   & &\times\prod_{i=1}^kN_{\Omega_i}(z_i,\overline{z_i})^{\mu_i(\alpha-d_0)-p_i}\Big(1-\frac{\|w\|^2}{\prod_{i=1}^k
   N_{\Omega_i}(z_i,\overline{z_i})^{\mu_i}}\Big)^{\alpha-n-1}\prod_{j=0}^kdm(z_j)  \\
   &=&\prod_{i=0}^k\dim\mathcal{P}^{(i)}_{\lambda_i}.
\end{eqnarray*}

Therefore, the Bergman kernel of $\mathcal{H}_{\alpha}$ can be
written as
\begin{equation}\label{e4.4}
  K_{\alpha}(Z;\overline{Z})=\sum_{{\ell(\lambda_i)\leq r_i\atop 0\leq
    i\leq k}}\frac{\prod_{i=0}^k\dim\mathcal{P}^{(i)}_{\lambda_i}}{<\prod_{j=0}^kK^{(j)}_{\lambda_j}(z_j;\overline{z_j})>}\prod_{j=0}^kK^{(j)}_{\lambda_j}(z_j;\overline{z_j}),
\end{equation}
where $<f>$ denotes integral
\begin{eqnarray}
\nonumber <f>  &=& \frac{\prod_{i=1}^k(\mu_i^{d_i}C_{\Omega_i})}{\pi^{n}}\int_{\left(\prod_{j=1}^k\Omega_j\right)^{{\mathbb{B}}^{d_0}}(\mu)}f(Z) \\
\nonumber   & &
\times\prod_{i=1}^kN_{\Omega_i}(z_i,\overline{z_i})^{\mu_i(\alpha-d_0)-p_i}\Big(1-\frac{\|w\|^2}{\prod_{i=1}^kN_{\Omega_i}(z_i,\overline{z_i})^{\mu_i}}\Big)
^{\alpha-n-1}dm(Z).
\end{eqnarray}

If $\mu_i\alpha-p_i>-1$ and $\alpha-n-1>-1$, namely
$\alpha>\max\{n,\frac{p_1-1}{\mu_1},\ldots,\frac{p_k-1}{\mu_k}\}$,
combining (see \cite{F})
\begin{equation}\label{eq}
    \int_{\Omega}K_{\lambda}(z,\overline{z})N_{\Omega}(z,\overline{z})^sdm(z)=\frac{\dim \mathcal{P}_{\lambda}}{(p+s)_{\lambda}}
    \int_{\Omega}N_{\Omega}(z,\overline{z})^sdm(z)
\end{equation}
for $ s>-1 $  and \eqref{1.2},  we have
\begin{eqnarray}
\nonumber   & & <\prod_{j=0}^kK^{(j)}_{\lambda_j}(z_j;\overline{z_j})> \\
\nonumber   &=&
\frac{\prod_{i=1}^k(\mu_i^{d_i}C_{\Omega_i})}{\pi^{n}}\prod_{i=1}^k\int_{\Omega_i}K^{(i)}_{\lambda_i}(z_i;\overline{z_i})N_{\Omega_i}(z_i,\overline{z_i})^{\mu_i(\alpha+|\lambda_0|)-p_i}
dm(z_i)\\
\nonumber   & &
\times \int_{{\mathbb{B}}^{d_0}} K^{(0)}_{\lambda_0}(w;\overline{w})(1-\|w\|^2)^{\alpha-n-1}dm(w)\\
\label{e4.5}
&=&\frac{\prod_{i=1}^k(\mu_i^{d_i}C_{\Omega_i}\chi_i(0)V(\Omega_i))\cdot V({\mathbb{B}}^{d_0})\chi_0(0)}{\pi^{n}\chi_0(\alpha-n-1)\prod_{i=1}^k\chi_i(\mu_i(\alpha+|\lambda_0|)-p_i)}
\frac{\prod_{i=0}^k\dim\mathcal{P}^{(i)}_{\lambda_i}}{(\alpha-d)_{\lambda_0}^{(0)}\prod_{i=1}^k(\mu_i(\alpha+|\lambda_0|))_{\lambda_i}^{(i)}},
\end{eqnarray}
where we use the fact $p_0=d_0+1$.

Combining  \eqref{e4.4}, \eqref{e4.5} and \eqref{1.7}, we get
\begin{eqnarray}
\nonumber   & & K_{\alpha}(Z;\overline{Z}) \\
\nonumber    &=&
  \sum_{{\ell(\lambda_i)\leq r_i\atop 0\leq i\leq k}}
  c\left(\prod_{i=1}^k\chi_i(\mu_i(\alpha+|\lambda_0|)-p_i)(\mu_i(\alpha+|\lambda_0|))_{\lambda_i}^{(i)}K^{(j)}_{\lambda_j}(z_j;\overline{z_j})\right)(\alpha-d)_{\lambda_0}^{(0)}
   K^{(0)}_{\lambda_0}(w;\overline{w})   \\
\nonumber   &=&
c\sum_{{\ell(\lambda_0)\leq r_0}}\prod_{i=1}^k \chi_i(\mu_i(\alpha+|\lambda_0|)-p_i)\frac{1}{N_{\Omega_i}(z_i,\overline{z_i})^{\mu_i(\alpha+|\lambda_0|)}}\cdot(\alpha-d)_{\lambda_0}^{(0)} K^{(0)}_{\lambda_0}(w,\overline{w})   \\
\nonumber   &=&
\left.\frac{c}{{\prod_{i=1}^kN_{\Omega_i}(z_i,\overline{z_i})^{\mu_i\alpha}}}\sum_{{\ell(\lambda_0)\leq
r_0}}\prod_{i=1}^k\chi_i(\mu_i(\alpha+t\frac{d}{dt})-p_i)(\alpha-d)_{\lambda_0}^{(0)}
K^{(0)}_{\lambda_0}\Big(\frac{tw}
{\prod_{i=1}^kN_{\Omega_i}(z_i,\overline{z_i})^{\mu_i}},\overline{w}\Big)\right|_{t=1}   \\
\nonumber
&=&\left.\frac{c}{{\prod_{i=1}^kN_{\Omega_i}(z_i,\overline{z_i})^{\mu_i\alpha}}}
\prod_{i=1}^k\chi_i(\mu_i(\alpha+t\frac{d}{dt})-p_i)\frac{1}{\left(1-\frac{t\|w\|^2}{\prod_{i=1}^kN_{\Omega_i}(z_i,\overline{z_i})^{\mu_i}}\right)^{\alpha-d}}\right|_{t=1},
\end{eqnarray}
where
$$c=\frac{\pi^{n}\chi_0(\alpha-n-1)}{\prod_{i=1}^k(\mu_i^{d_i}C_{\Omega_i}\chi_i(0)V(\Omega_i))V({\mathbb{B}}^{d_0})\chi_0(0)}.$$
Combining
$$V(\Omega_i)=\frac{\pi^{d_i}}{C_{\Omega_i}\chi_i(0)},~~V({\mathbb{B}}^{d_0})=\frac{\pi^{d_0}}{\chi_0(0)}~(\text{refer to} \cite{Hua, Ko})
\;\;\mbox{and} \;\; \chi_0(x)=(x+1)_{d_0},$$ we obtain
\eqref{eq2.26}.
\end{proof}

In order to simplify  \eqref{eq2.26}, we need Lemma \ref{Le:2.1.1} below.

\begin{Lemma}\textup{(see \cite{F})}\label{Le:2.1.1}{
Let $\varphi(x)$ be a polynomial in $x$ of degree $n$ and let $Z$ be
a matrix of order $m$. Assume $t$ is a real variable such that
$||tZ||<1$, {where $||Z||$ denotes the norm of  $Z$}. For a real
number $n_0$, take $x_0=-mn_0$. Then we have
\begin{equation}\label{2.4.0}
\varphi(t\frac{d}{dt})\frac{1}{{\det}(I-tZ)^{n_0}}=\frac{1}{{\det}(I-tZ)^{n_0}}\sum_{j=0}^n\frac{D^j\varphi(x_0)}{j!}\sum_{|\lambda|=j}\frac{|\lambda|!}{z_{\lambda}}n_0^{\ell(\lambda)}p_{\lambda}(\frac{1}{I-tZ}),
\end{equation}
where
\begin{equation*}
    \lambda=(1^{m_1(\lambda)}2^{m_2(\lambda)}\ldots),\quad m_i(\lambda) \geq 0,
\end{equation*}
\begin{equation*}
    |\lambda|:=\sum_{i}im_i(\lambda),\quad\ell(\lambda):=\sum_{i}m_i(\lambda),\quad z_{\lambda}:=\prod_{i}i^{m_i(\lambda)}m_i(\lambda)!,
\end{equation*}
\begin{equation*}
    p_{\lambda}(Z):=\prod_{i}(\textup{Tr}Z^i)^{m_i(\lambda)},\quad D^j\varphi(x_0)=\sum_{l=0}^j{j\choose l}(-1)^l\varphi(x_0-l).
\end{equation*}
 }\end{Lemma}

Combing Theorem \ref{Th:a.2} and Lemma \ref{Le:2.1.1}, we obtain the
explicit expression of the Bergman kernel $K_{\alpha}$ of the
weighted Hilbert space $\mathcal{H}_{\alpha}$ as follows.

\begin{Theorem}\label{Th:2.2}
Assume
\begin{equation}\label{eq2.29}
   \widetilde{\chi}(x):=\prod_{i=1}^k\chi_i(\mu_i x-p_i)\equiv \prod_{i=1}^k \prod_{j=1}^{r_i}\left(\mu_i
   x-p_i+1+(j-1)\frac{a_i}{2}\right)_{1+b_i+(r_i-j)a_i}.
\end{equation}
Let $D^j\widetilde{\chi}(x)$ be the  $j$-order difference of
$\widetilde{\chi}$ at $x$, that is
\begin{equation}\label{eq2.29.0}
D^j\widetilde{\chi}(x)=\sum_{l=0}^j{j\choose
l}(-1)^l\widetilde{\chi}(x-l).
\end{equation}
Then \eqref{eq2.26} can be rewritten as
\begin{eqnarray}
\nonumber          & & K_{\alpha}(z_1,\ldots,z_k,w;\overline{z_1},\ldots,\overline{z_k},\overline{w})  \\
\label{eq2.26.1}
                   &=&\frac{1}{\prod_{i=1}^k\mu_i^{d_i}}\prod_{i=1}^k\left(\frac{1}{N_{\Omega_i}(z_i,\overline{z_i})}\right)^{\mu_i\alpha}
                       \sum_{j=0}^{d}\frac{D^j\widetilde{\chi}(d)}{j!}\frac{(\alpha-n)_{j+d_0}}{\left(1-\frac{\|w\|^2}{\prod_{i=1}^kN_{\Omega_i}(z_i,\overline{z_i})^{\mu_i}}\right)^{\alpha-d+j}}.
\end{eqnarray}
\end{Theorem}

\begin{proof}[Proof] Let $x_0=d-\alpha$, $\varphi(x)=\prod_{i=1}^k\chi_i(\mu_i(\alpha+x)-p_i)$. From
$$\prod_{i=1}^k\chi_i(\mu_i(\alpha+x)-p_i)|_{x=x_0-l}=\prod_{i=1}^k\chi_i(\mu_i(d-j)-p_i)=\widetilde{\chi}(d-l),$$
 we have
 $$D^j\varphi(x)|_{x=x_0}=D^j\widetilde{\chi}(d).$$
 Using \eqref{2.4.0} and
 \begin{equation*}
    (x)_j=\sum_{|\lambda|=j}\frac{|\lambda|!}{z_{\lambda}}x^{\ell(\lambda)},
 \end{equation*}
 we have
\begin{equation}\label{eq2.31}
 \varphi(t\frac{d}{dt})\frac{1}{(1-tz)^{\alpha-d}}=\frac{1}{(1-tz)^{\alpha-d}}\sum_{j=0}^{d}\frac{D^j\widetilde{\chi}(d)}{j!}\frac{(\alpha-d)_j}{(1-tz)^j}.
\end{equation}
Combining
\begin{equation}\label{eq2.32}
    (\alpha-n)_{d_0}(\alpha-d)_j=(\alpha-n)_{d_0+j}
\end{equation}
and \eqref{eq2.31},  we get \eqref{eq2.26.1}.
\end{proof}

\setcounter{equation}{0}
\section{The Rawnsley's $\varepsilon$-function for $\big(\prod_{j=1}^k\Omega_j\big)^{{\mathbb{B}}^{d_0}}(\mu)$  with the canonical metric $g(\mu)$}

In this section we give the explicit expression of  the Rawnsley's
$\varepsilon$-function and the coefficients $a_1,a_2$ of its
expansion for the generalized Cartan-Hartogs domain
$\big(\big(\prod_{j=1}^k\Omega_j\big)^{{\mathbb{B}}^{d_0}}(\mu),
g(\mu)\big)$  with the canonical metric $g(\mu)$.

\begin{Theorem}\label{Th:2.3}{Let $\Omega_i$ be an irreducible bounded symmetric domain  in
$\mathbb{C}^{d_i}$ in its Harish-Chandra realization, and denote the
generic norm $N_{\Omega_i}(z_i,\overline{z_i})$, the dimension $d_i$
and the genus $p_i$ for $\Omega_i$ $(1\leq i \leq k)$. Set
$n=\sum_{j=0}^kd_j$, $d=\sum_{j=1}^kd_j$ and
$\alpha>\max\{n,\frac{p_1-1}{\mu_1},\ldots,\frac{p_k-1}{\mu_k}\}$.
Then the Rawnsley's $\varepsilon$-function associated to
$\big(\big(\prod_{j=1}^k\Omega_j\big)^{{\mathbb{B}}^{d_0}}(\mu),
g(\mu)\big)$ can be written as
\begin{equation}\label{eq2.37}
    \varepsilon_{\alpha}(z_1,\ldots,z_k,w)=\frac{1}{\prod_{i=1}^k\mu_i^{d_i}}\sum_{j=0}^{d}\frac{D^j\widetilde{\chi}(d)}{j!}\Big(1-\frac{\|w\|^2}
    {\prod_{i=1}^kN_{\Omega_i}(z_i,\overline{z_i})^{\mu_i}}\Big)^{d-j}(\alpha-n)_{j+d_0}
\end{equation}
$($see \eqref{eq2.29} and \eqref{eq2.29.0} for the definition of the
functions $\widetilde{\chi}(x)$ and $D^j\widetilde{\chi}(x)$
respectively$)$. }\end{Theorem}

\begin{proof}[Proof]
By \eqref{eq2.26.1} and
\begin{equation*}
   \varepsilon_{\alpha}(z_1,\ldots,z_k,w):= e^{-\alpha\Phi(z_1,\ldots,z_k,w)} K_{\alpha}(z_1,\ldots,z_k,w;\overline{z_1},\ldots,\overline{z_k},\overline{w}),
\end{equation*}
we obtain \eqref{eq2.37}.
\end{proof}

\begin{Corollary}\label{Co:2.1}{For $k\geq 2$, the coefficients $a_1$ and $a_2$ of the expansion of the Rawnsley's $\varepsilon$-function $\varepsilon_{\alpha}$, that is, the coefficients of $\alpha^{n-1}$ and $ \alpha^{n-2}$ in \eqref{eq2.37} respectively,  are given by
\begin{equation}\label{eq2.38}
   a_1(z_1,\ldots,z_k,w)=\frac{1}{\prod_{i=1}^k\mu_i^{d_i}}\frac{D^{d-1}\widetilde{\chi}(d)}{(d-1)!}
   \Big(1-\frac{\|w\|^2}{\prod_{i=1}^kN_{\Omega_i}(z_i,\overline{z_i})^{\mu_i}}\Big)-\frac{n(n+1)}{2},
\end{equation}
\begin{eqnarray}
\nonumber  a_2(z_1,\ldots,z_k,w) &=& \frac{1}{\prod_{i=1}^k\mu_i^{d_i}}\frac{D^{d-2}\widetilde{\chi}(d)}{(d-2)!}
\Big(1-\frac{\|w\|^2}{\prod_{i=1}^kN_{\Omega_i}(z_i,\overline{z_i})^{\mu_i}}\Big)^2  \\
\nonumber           &
&-\frac{1}{\prod_{i=1}^k\mu_i^{d_i}}\frac{D^{d-1}\widetilde{\chi}(d)}{(d-1)!}\Big(\frac{n(n+1)}{2}-1\Big)
\Big(1-\frac{\|w\|^2}{\prod_{i=1}^kN_{\Omega_i}(z_i,\overline{z_i})^{\mu_i}}\Big) \\
\label{eq2.39} & &+\frac{1}{24}(n-1)n(n+1)(3n+2).
\end{eqnarray}
 }\end{Corollary}

\begin{proof}[Proof] Write
\begin{equation}\label{eq2.40}
    (\alpha-n)_{d_0+l}=\sum_{j=0}^{d_0+l}c_{d_0+l,j}\alpha^j.
\end{equation}
Substituting \eqref{eq2.40} into \eqref{eq2.37},  we obtain
\begin{equation}\label{eq2.41}
   \varepsilon_{\alpha}(z_1,\ldots,z_k,w)=\sum_{j=0}^{n}\alpha^j\sum_{l=\max(j-d_0,0)}^d
  \frac{c_{d_0+l,j}}{\prod_{i=1}^k\mu_i^{d_i}}
  \frac{D^l\widetilde{\chi}(d)}{l!}\Big(1-\frac{\|w\|^2}{\prod_{i=1}^kN_{\Omega_i}(z_i,\overline{z_i})^{\mu_i}}\Big)^{d-l},
\end{equation}
which implies
\begin{equation}\label{eq2.42}
a_j(z_1,\ldots,z_k,w)=\sum_{l=\max(d-j,0)}^d
\frac{c_{d_0+l,n-j}}{\prod_{i=1}^k\mu_i^{d_i}}\frac{D^l\widetilde{\chi}(d)}{l!}\Big(1-\frac{\|w\|^2}{\prod_{i=1}^kN_{\Omega_i}(z_i,\overline{z_i})
^{\mu_i}}\Big)^{d-l}.
\end{equation}
From
\begin{eqnarray*}
  (\alpha-n)_{n}   &=& \prod_{k=1}^{n}(\alpha-k), \\
  (\alpha-n)_{n-1} &=& \prod_{k=2}^{n}(\alpha-k), \\
  (\alpha-n)_{n-2} &=& \prod_{k=3}^{n}(\alpha-k),
\end{eqnarray*}
we have
\begin{equation}\label{eq2.44}
c_{n-1,n-1}=c_{n-2,n-2}=1,
\end{equation}
\begin{equation}\label{eq2.45}
c_{n,n-1}=-\sum_{k=1}^{n}k=-\frac{n(n+1)}{2},
\end{equation}
\begin{equation}\label{eq2.46}
c_{n-1,n-2}=-\sum_{k=2}^{n}k=-\frac{n(n+1)}{2}+1,
\end{equation}
\begin{eqnarray}
\nonumber  c_{n,n-2} &=& \sum_{1\leq i<j\leq
                                 n}ij=\frac{1}{2}\left\{\left(\sum_{k=1}^{n}k\right)^2-\sum_{k=1}^{n}k^2\right\} \\
\label{eq2.47}&=&\frac{1}{24}(n-1)n(n+1)(3n+2).
\end{eqnarray}
Substituting \eqref{eq2.44}, \eqref{eq2.45}, \eqref{eq2.46} and
\eqref{eq2.47} into \eqref{eq2.42}, we obtain \eqref{eq2.38} and
\eqref{eq2.39}.
\end{proof}

In order to calculate $D^{d-1}\widetilde{\chi}$ and
$D^{d-2}\widetilde{\chi}$, we need Lemma \ref{Le:2.8} and
\ref{Le:2.9} below.

\begin{Lemma}\textup{(see \cite{FT})}\label{Le:2.8}{
Write $\prod_{j=1}^r\left(\mu
   x-p+1+(j-1)\frac{a}{2}\right)_{1+b+(r-j)a}=\sum_{j=0}^dc_jx^{d-j}$. Then
\begin{equation}\label{eq2.49}
    c_0=\mu^d,
\end{equation}
\begin{equation}\label{eq2.50}
   c_1=-\frac{1}{2}\mu^{d-1}dp,
\end{equation}
\begin{equation}\label{eq2.51}
  c_2= \frac{1}{2}\mu^{d-2}R(\Omega),
\end{equation}
where
\begin{eqnarray}
\nonumber R(\Omega)  &=&\frac{d^2p^2}{4}-\frac{r(p-1)p(2p-1)}{6}+\frac{r(r-1)a(3p^2-3p+1)}{12} \\
\label{eq2.51.1}&&-\frac{(r-1)r(2r-1)a^2(p-1)}{24}+\frac{r^2(r-1)^2a^3}{48}.
\end{eqnarray}
 }\end{Lemma}

\begin{Lemma}\textup{(see \cite{FT})}\label{Le:2.9}{
For any  polynomial $f(x)$ in real variable $x$, take
$Df(x):=f(x)-f(x-1)$. Let $A_d=D^{d-1}x^d$, $B_d=D^{d-2}x^d$. Then
we have
\begin{eqnarray}
\label{eq2.56} A_d  &=& \frac{d!}{2}(2x-d+1)\quad  (d\geq 1),  \\
\label{eq2.57} B_d  &=&
\frac{d!}{24}\left\{12x^2-12(d-2)x+3d^2-11d+10\right\}\quad (d\geq
2).
\end{eqnarray}
 }\end{Lemma}

Lemma \ref{Le:2.8} and Lemma \ref{Le:2.9} imply the following results.
\begin{Lemma}\label{Le:2.10}{
Suppose that $k=2$, $d=d_1+d_2$, $D^{d-1}\widetilde{\chi}(d)$ and
$D^{d-2}\widetilde{\chi}(d)$ are defined by \eqref{eq2.29} and \eqref{eq2.29.0}. Then
we have
\begin{equation}\label{eq2.61}
   \frac{1}{\mu_1^{d_1}\mu_2^{d_2}} \frac{D^{d-1}\widetilde{\chi}(d)}{(d-1)!}=\frac{1}{2}\left\{d(d+1)-\left(\frac{d_1p_1}
   {\mu_1}+\frac{d_2p_2}{\mu_2}\right)\right\},
\end{equation}
\begin{eqnarray}
\nonumber \frac{1}{\mu_1^{d_1}\mu_2^{d_2}}     \frac{D^{d-2}\widetilde{\chi}(d)}{(d-2)!} &=& \frac{1}{4}\left\{\frac{1}{6}(d-1)d(d+1)(3d+10)-(d-1)(d+2)\left(\frac{d_1p_1}{\mu_1}+\frac{d_2p_2}{\mu_2}\right)\right. \\
\label{eq2.62}   & & +
  \left.2\frac{R(\Omega_1)}{\mu_1^2}+\frac{d_1d_2p_1p_2}{\mu_1\mu_2}+2\frac{R(\Omega_2)}{\mu_2^2}\right\},
\end{eqnarray}
where
\begin{eqnarray}
\nonumber R(\Omega_i) &=& \frac{d_i^2p_i^2}{4}-\frac{r_i(p_i-1)p_i(2p_i-1)}{6}+\frac{r_i(r_i-1)a_i(3p_i^2-3p_i+1)}{12} \\
 \label{eq2.63}           & & -\frac{(r_i-1)r_i(2r_i-1)a_i^2(p_i-1)}{24}+\frac{r_i^2(r_i-1)^2a_i^3}{48}\;\;\;(1\leq i\leq 2).
\end{eqnarray}
 }\end{Lemma}

\begin{proof}[Proof]
Let $\widetilde{\chi}(x)=c_0x^d+c_1x^{d-1}+c_2x^{d-2}+\cdots+c_n$.
From Lemma \ref{Le:2.8}, we get
\begin{equation}\label{eq2.64.1}
  c_0=\mu_1^{d_1}\mu_2^{d_2},
\end{equation}
\begin{equation}\label{eq2.64.2}
  c_1=-\frac{1}{2}\mu_1^{d_1}\mu_2^{d_2}\left(\frac{d_1p_1}{\mu_1}+\frac{d_2p_2}{\mu_2}\right),
\end{equation}
\begin{equation}\label{eq2.64.3}
  c_2=\frac{1}{4}\mu_1^{d_1}\mu_2^{d_2}\left(2\frac{R(\Omega_1)}{\mu_1^2}+\frac{d_1d_2p_1p_2}{\mu_1\mu_2}+2\frac{R(\Omega_2)}{\mu_2^2}\right),
\end{equation}
where $R(\Omega_i)$ is defined by \eqref{eq2.63} $(1\leq i\leq 2)$.
By \begin{equation}\label{eq2.64}
    \left\{\begin{array}{cl}
          D^{d-1}\widetilde{\chi}(x) =   & c_0A_d(x)+c_1(d-1)!,\\
          D^{d-2}\widetilde{\chi}(x) =   &
          c_0B_d(x)+c_1A_{d-1}(x)+c_2(d-2)!,
           \end{array}
    \right.
\end{equation}
letting $x=d$ and substituting \eqref{eq2.64.1}, \eqref{eq2.64.2},
\eqref{eq2.64.3}, \eqref{eq2.56} and \eqref{eq2.57} into
\eqref{eq2.64}, we obtain \eqref{eq2.61} and \eqref{eq2.62}.
\end{proof}

\setcounter{equation}{0}
\section{The proof of Theorem \ref{Th:2}}

(i) From the proof of Theorem \ref{Th:a.2}, we know that the space
$\mathcal{H}_{\alpha}\neq \{0\}$ if and only if
$$\alpha>\max\{n,\frac{p_1-1}{\mu_1},\ldots,\frac{p_k-1}{\mu_k}\}.$$

Using \eqref{eq2.37}, we obtain that
$\varepsilon_{\alpha}(z_1,\ldots,z_k,w)$ is a constant with respect
to $(z_1,\ldots,z_k,w)$ if and only if
\begin{equation*}
  D^j\widetilde{\chi}(d)=0 \;\;(0\leq j\leq d-1).
\end{equation*}
This is equivalent to
\begin{equation*}
  \widetilde{\chi}(x)= \sum_{j=0}^d\frac{D^j\widetilde{\chi}(d)}{j!}(x-d)_j=\frac{D^d\widetilde{\chi}(d)}{d!}(x-d)_d=
  \prod_{j=1}^k\mu_j^{d_j}\cdot\prod_{j=1}^d(x-j),
\end{equation*}
which implies \eqref{eq1.7} by the definition of
$\widetilde{\chi}(x)$ (see \eqref{eq2.29}).

(ii) Since there is no multiple divisor for the polynomial
$\prod_{j=1}^d(x-j)$, by \eqref{eq1.7}, we obtain that each
polynomial
$$\chi_i(\mu_ix-p_i)=\prod_{j=1}^{r_i}\left(\mu_ix-p_i+1+(j-1)\frac{a_i}{2}\right)_{1+b_i+(r_i-j)a_i}$$
has no multiple divisor, namely each $\chi_i(x)$ has no multiple
divisor.

(1) For the irreducible bounded symmetric domain $\Omega_I(m,n)\;
(1\leq m\leq n)$, its rank $r=m$, the characteristic multiplicities
$a=2, b=n-m$. Since
\begin{equation}\label{eq5.1}
\chi(x)=\prod_{j=1}^m(x+j)_{m+n+1-2j},
\end{equation}
it is easy to see that $\chi(x)$ has no multiple divisor iff
$r=m=1$.

(2) For the irreducible bounded symmetric domain  $\Omega_{II}(2n)\;
(n\geq 3)$, its rank $r=n$, the characteristic multiplicities $a=4,
b=0$. The polynomial
\begin{equation}\label{eq5.2}
 \chi(x)=\prod_{j=1}^n(x-1+2j)_{4n+1-4j}
\end{equation}
has multiple divisors $x+3$.

For the irreducible bounded symmetric domain
$\Omega_{II}(2n+1)\;(n\geq 2) $, its rank $r=n$, the characteristic
multiplicities $a=4, b=2$. The polynomial
\begin{equation}\label{eq5.3}
\chi(x)=\prod_{j=1}^n(x-1+2j)_{4n+3-4j}
\end{equation}
has multiple divisors $x+3$.

 (3) For the irreducible bounded symmetric domain $\Omega_{III}(n)\;(n\geq 2) $, its rank
$r=n$, the characteristic multiplicities  $a=1, b=0$. When $n\geq
3$,
\begin{equation}\label{eq5.4}
\chi(x)=\prod_{j=1}^n\left(x+\frac{1+j}{2}\right)_{n+1-j}
\end{equation}
has multiple divisors $x+2$. When $n=2$,
\begin{equation}\label{eq5.5}
\chi(x)=(x+1)(x+2)\left(x+\frac{3}{2}\right)
\end{equation}
has not any multiple divisor.

(4) For the irreducible bounded symmetric domain
$\Omega_{IV}(n)\;(n\geq 5) $, its rank
 $r=2$, the characteristic multiplicities $a=n-2, b=0$. From
\begin{equation}\label{eq5.6}
\chi(x)=(x+1)_{n-1}\left(x+\frac{n}{2}\right),
\end{equation}
we obtain that $\chi(x)$ has no multiple divisor if and only if $n$
is odd.

(5) For the irreducible bounded symmetric domain
$\Omega_{\textrm{V}}(16) $, its rank $r=2$, the characteristic
multiplicities $a=6, b=4$. By
\begin{equation}\label{eq5.7}
 \chi(x)=(x+1)_{11}(x+4)_5,
\end{equation}
we get that $\chi(x)$ has multiple divisors.

 (6) For the irreducible bounded symmetric domain  $\Omega_{\textrm{VI}}(27) $,
 its rank
$r=3$, the characteristic multiplicities $a=8, b=0$. The polynomial
\begin{equation}\label{eq5.8}
  \chi(x)=(x+1)_{17}(x+5)_{9}(x+9)
\end{equation}
has multiple divisors.

 In summary, we have that if the metric
$\alpha g(\mu)$  on
$\left(\prod_{j=1}^k\Omega_j\right)^{{\mathbb{B}}^{d_0}}(\mu)$  is
balanced, then $\Omega_i$ must be  biholomorphic to one of
$\Omega_{I}(1,n)$, $\Omega_{III}(2)$ or $\Omega_{IV}(m)~(m\geq 5~$
and $m \text{ are odd} )$.

(iii) Using Theorem \ref{Th:2}(i), it is easy to show that the
metrics $\alpha g(\mu)$ on
$$\left({\mathbb{B}}^d\times {\mathbb{B}}\right)^{{\mathbb{B}}^{d_0}}\left(1,\frac{1}{d+1}\right)$$ and $$\left({\mathbb{B}}\times\Omega_{III}(2)\right)^{{\mathbb{B}}^{d_0}}\left(1,\frac{1}{2}\right)$$
are balanced.

Now suppose that $\alpha g(\mu)$ on $(\Omega_1\times
\Omega_2)^{\mathbb{B}^{d_0}}(\mu)$ is balanced, by Theorem
\ref{Th:2}(ii),
 we have
that $\Omega_j$ $(j=1,2)$ must be biholomorphic to one of
$\Omega_{I}(1,n)$, $\Omega_{III}(2)$ or $\Omega_{IV}(m)~(m\geq 5$
and $m$ are odd).

(1) When $(\Omega_1,\Omega_2)=(\mathbb{B}^{d_1},\mathbb{B}^{d_2})$,
using \eqref{eq1.7}, we get
\begin{equation}\label{eq6.1}
\prod_{j=1}^{d_1}\left(x-\frac{j}{\mu_1}\right)\cdot
\prod_{j=1}^{d_2}\left(x-\frac{j}{\mu_2}\right)
=\prod_{j=1}^{d_1+d_2}(x-j).
\end{equation}

This means that  $\frac{j}{\mu_1},\; \frac{l}{\mu_2}$  $(1\leq j\leq
d_1, 1\leq l\leq d_2, j,l\in \mathbb{N})$ are zeros of the
polynomial $\prod_{j=1}^{d_1+d_2}(x-j)$. Thus $\frac{1}{\mu_1}$ and
$\frac{1}{\mu_2}$ are integers. Assume $d_1\geq d_2$. Since
$\frac{d_1}{\mu_1}$ is a zero of the polynomial
$\prod_{j=1}^{d_1+d_2}(x-j)$, we get that $\frac{d_1}{\mu_1}\leq
d_1+d_2\leq 2d_1,$ and thus we have $1\leq \frac{1}{\mu_1}\leq 2$.

If $\frac{1}{\mu_1}=1$, by \eqref{eq6.1}, it follows that
$$\frac{d_2}{\mu_2}=d_1+d_2,\;\; \frac{1}{\mu_2}=d_1+1.$$
So $(d_1-2)(d_2-1)=0$. If $d_2\geq 2$, then $d_1=d_2=2$ and
$\frac{1}{\mu_2}=3$. But, in this case, the equation \eqref{eq6.1}
is not valid. Therefore, if $\frac{1}{\mu_1}=1$, then $d_2=1$ and
$\frac{1}{\mu_2}=d_1+1$.

If $\frac{1}{\mu_1}=2$, since the number of even zeros of the left
side of \eqref{eq6.1} is larger than or equal to $d_1$ and the
number of even zeros is equal to $\left[\frac{d_1+d_2}{2}\right]$
for the right side of \eqref{eq6.1}, where $[n]$ denotes the
greatest integer which is less than or equal to $n$,  we get
$$\frac{d_1+d_2}{2}\geq \left[\frac{d_1+d_2}{2}\right]\geq d_1.$$
This yilds $d_2\geq d_1$. From the assumption $d_1\geq d_2$, we have
$d_1=d_2$. If $d_1=d_2>1$, since numbers of odd zeros of both sides
of \eqref{eq6.1} are not equal, this leads to a contradiction. If
$d_1=d_2=1$, using \eqref{eq6.1}, we obtain $\frac{1}{\mu_2}=1$.

(2) When $(\Omega_1,\Omega_2)=(\mathbb{B}^{d_1},\Omega_{III}(2))$,
applying \eqref{eq1.7}, we obtain
\begin{equation}\label{eq6.2}
  \prod_{j=1}^{d_1}\left(x-\frac{j}{\mu_1}\right)\cdot\left(x-\frac{1}{\mu_2}\right)\left(x-\frac{2}{\mu_2}\right)\left(x-\frac{3}{2\mu_2}\right)=\prod_{j=1}^{d_1+3}(x-j).
\end{equation}
This means that $\frac{j}{\mu_1}(1\leq j\leq
d_1),\frac{1}{\mu_2},\frac{2}{\mu_2}$ and $\frac{3}{2\mu_2}$ are
zeros of the polynomial $\prod_{j=1}^{d_1+3}(x-j)$. So
$\frac{1}{\mu_1}, \frac{1}{\mu_2}$ and $\frac{3}{2\mu_2}$ are
integers, and thus $\frac{1}{\mu_2}=2t$ (i.e., $\frac{1}{\mu_2}$ is
even). Since $ \prod_{j=1}^{d_1}\left(x-\frac{j}{\mu_1}\right) $ is
divisible by $x-1$, we get $\frac{1}{\mu_1}=1$ and so
\begin{equation*}
  (x-2t)(x-3t)(x-4t)=(x-d_1-1)(x-d_1-2)(x-d_1-3).
\end{equation*}
Thus we have $t=1,\; d_1=1$.

(3) When $(\Omega_1,\Omega_2)=(\mathbb{B}^{d_1},\Omega_{IV}(d_2))$
(where $d_2\;\geq 5$ are odd), \eqref{eq1.7} implies
\begin{equation}\label{eq6.3}
 \prod_{j=1}^{d_1}\left(x-\frac{j}{\mu_1}\right)\cdot \prod_{j=1}^{d_2-1}\left(x-\frac{j}{\mu_2}\right)\cdot \left(x-\frac{d_2}{2\mu_2}\right)=
 \prod_{j=1}^{d_1+d_2}(x-j).
\end{equation}
This implies that $\frac{1}{\mu_1}, \frac{1}{\mu_2}$ and
$\frac{d_2}{2\mu_2}$ are integers (where $d_2\;\geq 5$ are odd). So
$\frac{1}{\mu_1}$ is an integer and $\frac{1}{\mu_2}$ is even. From
the number of even zeros of the left side of \eqref{eq6.3} is
greater than or equal to $\left[\frac{d_1}{2}\right]+d_2-1$ and the
number of even zeros is equal to $\left[\frac{d_1+d_2}{2}\right]$
for the right side of \eqref{eq6.3}, we have
\begin{equation*}
 \left[\frac{d_1+d_2}{2}\right]\geq \left[\frac{d_1}{2}\right]+d_2-1.
\end{equation*}
Since
\begin{equation*}
  \frac{d_1+d_2}{2}-\left(\frac{d_1}{2}-1\right)\geq \left[\frac{d_1+d_2}{2}\right]-\left[\frac{d_1}{2}\right],
\end{equation*}
we have $d_2\leq 4$, this conflicts with $d_2\geq 5$.

(4) When $(\Omega_1,\Omega_2)=(\Omega_{III}(2),\Omega_{III}(2))$,
by \eqref{eq1.7}, we obtain
\begin{equation}\label{eq6.4}
 \left(x-\frac{1}{\mu_1}\right)\left(x-\frac{2}{\mu_1}\right)\left(x-\frac{3}{2\mu_1}\right) \left(x-\frac{1}{\mu_2}\right)\left(x-\frac{2}{\mu_2}\right)\left(x-\frac{3}{2\mu_2}\right)=\prod_{j=1}^{6}(x-j).
\end{equation}
This implies that $\frac{1}{\mu_1}, \frac{1}{\mu_2},
\frac{3}{2\mu_1}$ and $\frac{3}{2\mu_2}$ are integers. So
$\frac{1}{\mu_1}$ and $\frac{1}{\mu_2}$ are even. Since the number
of even zeros of the left side of \eqref{eq6.4} is greater than or
equal to 4 and the number of even zeros of the right side of
\eqref{eq6.4}  is equal to 3, this leads to a contradiction.

(5) When $(\Omega_1,\Omega_2)=(\Omega_{III}(2),\Omega_{IV}(d_2))$,
$d_2\geq 5$ and  $d_2$ is odd,  by \eqref{eq1.7}, we have
\begin{equation}\label{eq6.5}
   \left(x-\frac{1}{\mu_1}\right)\left(x-\frac{2}{\mu_1}\right)\left(x-\frac{3}{2\mu_1}\right)\cdot
   \prod_{j=1}^{d_2-1}\left(x-\frac{j}{\mu_2}\right)\cdot \left(x-\frac{d_2}{2\mu_2}\right)= \prod_{j=1}^{3+d_2}(x-j).
\end{equation}
Then $\frac{1}{\mu_1},\frac{3}{2\mu_1},\frac{1}{\mu_2}$ and
$\frac{d_2}{2\mu_2}$ are zeros of $ \prod_{j=1}^{3+d_2}(x-j)$. So
$\frac{1}{\mu_1}$ and $\frac{1}{\mu_2}$ are even. In view of the
number of even zeros of the left side of \eqref{eq6.5} is greater
than or equal to $d_2+1$ and the number of even zeros of the right
side of \eqref{eq6.5} is equal to $\left[\frac{3+d_2}{2}\right]$, we
get
\begin{equation*}
  \frac{3+d_2}{2}\geq \left[\frac{3+d_2}{2}\right]\geq d_2+1,
\end{equation*}
which means  $d_2\leq 1$, which conflicts with $d_2\geq 5$.

(6) When $(\Omega_1,\Omega_2)=(\Omega_{IV}(d_1),\Omega_{IV}(d_2))$,
$d_1\geq 5$, $d_2\geq 5$ and $d_1,d_2$ are odd, by \eqref{eq1.7}, we
have
\begin{equation}\label{eq6.6}
  \prod_{j=1}^{d_1-1}\left(x-\frac{j}{\mu_1}\right)\cdot \left(x-\frac{d_1}{2\mu_1}\right)\cdot \prod_{j=1}^{d_2-1}\left(x-\frac{j}{\mu_2}\right)
  \cdot\left(x-\frac{d_2}{2\mu_2}\right)= \prod_{j=1}^{d_1+d_2}(x-j).
\end{equation}
Then $\frac{1}{\mu_1},\frac{d_1}{2\mu_1},\frac{1}{\mu_2}$ and
$\frac{d_2}{2\mu_2}$ are zeros of $ \prod_{j=1}^{d_1+d_2}(x-j)$. So
$\frac{1}{\mu_1}$ and $\frac{1}{\mu_2}$ are even. Since the number
of even zeros of the left side of \eqref{eq6.6} is greater than or
equal to $d_1+d_2-2$ and the number of even zeros of the right side
of \eqref{eq6.6} is equal to $\left[\frac{d_1+d_2}{2}\right]$, we
have
\begin{equation*}
  \frac{d_1+d_2}{2}\geq \left[\frac{d_1+d_2}{2}\right]\geq d_1+d_2-2.
\end{equation*}
Therefore $d_1+d_2\leq 4$, which conflicts with $d_1\geq 5$ and
$d_2\geq 5$.

Combing the above results (1)-(6), we obtain that if $\alpha g(\mu)$
on $(\Omega_1\times \Omega_2)^{\mathbb{B}^{d_0}}(\mu)$ is balanced,
then $(\Omega_1\times \Omega_2)^{\mathbb{B}^{d_0}}(\mu)$ is
biholomorphic to
$$\left({\mathbb{B}}^d\times
{\mathbb{B}}\right)^{{\mathbb{B}}^{d_0}}\left(1,\frac{1}{d+1}\right)\;\;\text{or}\;\;
\left({\mathbb{B}}\times\Omega_{III}(2)\right)^{{\mathbb{B}}^{d_0}}\left(1,\frac{1}{2}\right).$$

 \setcounter{equation}{0}
\section{The proof of Theorem \ref{Th:1}}

In order to proof Theorem \ref{Th:1}, we need Lemma \ref{Le:3.1} and
\ref{Le:3.2} below.

\begin{Lemma}\label{Le:3.1}{Let $\Omega$ be an irreducible bounded symmetric domain in
$\mathbb{C}^{d}$ in its Harish-Chandra realization, and denote the
rank $r$, the characteristic multiplicities $a, b$, the dimension
$d$, the genus $p$, and the Hua polynomial $\chi$ of  $\Omega$. Then
\begin{equation}\label{eq4.1}
  \frac{1}{2}-\frac{8}{3(4d+1)}< S(\Omega)\leq \frac{1}{2}-\frac{1}{2d},
\end{equation}
where $S(\Omega)=\frac{2R(\Omega)}{d^2p^2}$ (see \eqref{eq2.51.1}
for $R(\Omega)$).}\end{Lemma}

\begin{proof}[Proof]
Let $x_1,x_2,\ldots,x_d$ be zeros of the polynomial $\chi(x-p)$. By
\eqref{1.3}, we know that $x_1,x_2,\ldots,x_d$ are real numbers.
From Lemma \ref{Le:2.8}, we have
\begin{equation}\label{eq4.2}
  \sum_{i=1}^dx_i=\frac{1}{2}dp,
\end{equation}
\begin{eqnarray}
\nonumber \sum_{i=1}^dx_i^2   &=& \frac{r(p-1)p(2p-1)}{6}-\frac{r(r-1)a(3p^2-3p+1)}{12} \\
\label{eq4.3}   & &
+\frac{(r-1)r(2r-1)a^2(p-1)}{24}-\frac{r^2(r-1)^2a^3}{48},
\end{eqnarray}
and
\begin{equation}\label{eq4.4}
 S(\Omega)=\frac{1}{2}-\frac{2\sum_{i=1}^dx_i^2}{d^2p^2}.
\end{equation}

From \eqref{eq4.2} and
\begin{equation*}
 \left(\sum_{i=1}^dx_i\right)^2\leq d\sum_{i=1}^dx_i^2,
\end{equation*}
we obtain
\begin{equation}\label{eq4.5}
 \frac{2\sum_{i=1}^dx_i^2}{d^2p^2}\geq \frac{1}{2d},
\end{equation}
which implies
\begin{equation*}
 S(\Omega)\leq \frac{1}{2}-\frac{1}{2d}.
\end{equation*}

Now we show
$$\frac{1}{2}-\frac{8}{3(4d+1)}< S(\Omega).$$
This is equivalent to
\begin{eqnarray}
\nonumber   &&  3(4d+1)\left\{\frac{r(p-1)p(2p-1)}{6}-\frac{r(r-1)a(3p^2-3p+1)}{12}\right. \\
\label{eq4.6}   & &
\left.+\frac{(r-1)r(2r-1)a^2(p-1)}{24}-\frac{r^2(r-1)^2a^3}{48}\right\}-4d^2p^2<
0.
\end{eqnarray}

We now calculate the left side of \eqref{eq4.6} by using the
classification of irreducible bounded symmetric domains.

(1) For the irreducible bounded symmetric domain $\Omega_I(m,n)\;
(1\leq m\leq n)$, its  rank $r=m$, the characteristic multiplicities
$a=2, b=n-m$, the dimension $d=mn$, the genus $p=m+n$.
\begin{equation}\label{eq4.7}
 \text{L.H.S of } \eqref{eq4.6}= -\frac{1}{2}mn(4m^2n^2-2m^2-2n^2+
 mn+1)\leq -\frac{1}{2}mn.
\end{equation}

(2) For the irreducible bounded symmetric domain  $\Omega_{II}(2n)\;
(n\geq 3) $, its rank $r=n$, the characteristic multiplicities $a=4,
b=0$, the dimension $d=n(2n-1)$, the genus $p=2(2n-1)$.
\begin{equation}\label{eq4.8}
 \text{L.H.S of } \eqref{eq4.6}=- n(32n^5 - 80n^4 + 60n^3 - 4n^2 - 9n +
 2)\leq -16n^5.
\end{equation}

For the irreducible bounded symmetric domain
$\Omega_{II}(2n+1)\;(n\geq 2) $, its  rank $r=n$, the characteristic
multiplicities $a=4, b=2$, the dimension $d=n(2n+1)$, the genus
$p=4n$.
\begin{equation}\label{eq4.9}
\text{L.H.S of } \eqref{eq4.6}= -n(2n + 1)(16n^4 - 10n^2 + 3n +
1)\leq -6n^5(2n+1).
\end{equation}

 (3) For the irreducible bounded symmetric domain $\Omega_{III}(n)\;(n\geq 2) $, its rank
$r=n$, the characteristic multiplicities $a=1, b=0$, the dimension
$d=n(n+1)/2$, the genus $p=n+1$.
\begin{equation}\label{eq4.10}
\text{L.H.S of } \eqref{eq4.6}= -\frac{1}{16}n(n + 1)(2n^4 + 8n^3 +
7n^2 - 5n - 4)\leq -\frac{1}{8}n^5(n + 1).
\end{equation}

(4) For the irreducible bounded symmetric domain
$\Omega_{IV}(n)\;(n\geq 5) $, its rank $r=2$, the characteristic
multiplicities $a=n-2, b=0$, the dimension $d=n$, the genus $p=n$.
\begin{equation}\label{eq4.11}
 \text{L.H.S of } \eqref{eq4.6}=-\frac{1}{4} n( 8n^2 - 5n - 2)\leq -\frac{1}{4} n^3.
\end{equation}

(5) For the irreducible bounded symmetric domain
$\Omega_{\textrm{V}}(16) $, its rank $r=2$, the characteristic
multiplicities $a=6, b=4$, the dimension $d=16$, the genus $p=12$.
\begin{equation}\label{eq4.12}
 \text{L.H.S of } \eqref{eq4.6}=-11736 .
\end{equation}

 (6) For the irreducible bounded symmetric domain  $\Omega_{\textrm{VI}}(27) $, its  rank
$r=3$, the characteristic multiplicities $a=8, b=0$, the dimension
$d=27$, the genus $p=18$.
\begin{equation}\label{eq4.13}
 \text{L.H.S of } \eqref{eq4.6}=-76599 .
\end{equation}

Combing the above (1)-(7), it follows that \eqref{eq4.6} holds. This
completes the proof.
\end{proof}

\begin{Lemma}\label{Le:3.2}{Let $\Omega_i$ be an irreducible bounded symmetric domain in
$\mathbb{C}^{d_i}$ in its Harish-Chandra realization, and denote the
rank $r_i$, the characteristic multiplicities $a_i, b_i$, the
dimension $d_i$ and the genus $p_i$ of $\Omega_i$, $1\leq i\leq 2$.
Assume that
\begin{equation}\label{eq4.14}
 d:=d_1+d_2,\;\;  A:=d(d+1),\;\;B:=d(d+1)\left\{(d-1)(d+2)-\frac{1}{6}(d-1)(3d+10)\right\}
\end{equation}
and
\begin{equation}\label{eq4.15}
  S(\Omega_1):=\frac{2R(\Omega_1)}{d_1^2p_1^2},\;\; S(\Omega_2):=\frac{2R(\Omega_2)}{d_2^2p_2^2}.
\end{equation}
Then
\begin{equation}\label{eq4.16}
  S(\Omega_1)+S(\Omega_2)<\frac{2B}{A^2}<1
\end{equation}
and
\begin{equation}\label{eq4.17}
  \frac{1}{1-2S(\Omega_1)}+\frac{1}{1-2S(\Omega_2)}>\frac{1}{1-\frac{2B}{A^2}}.
\end{equation}
 }\end{Lemma}

\begin{proof}[Proof]
By
\begin{equation*}
  d^2=(d_1+d_2)^2\geq 4d_1d_2,
\end{equation*}
we get
\begin{equation}\label{eq4.19}
  \frac{1}{2}\left(\frac{1}{d_1}+\frac{1}{d_2}\right)>\frac{2}{3}\frac{2d+1}{d(d+1)}.
\end{equation}
Using \eqref{eq4.1} and
\begin{equation}\label{eq4.18}
  \frac{B}{A^2}=\frac{1}{2}-\frac{2d+1}{3d(d+1)},
\end{equation}
 we obtain
\begin{equation*}
  S(\Omega_1)+S(\Omega_2)\leq 1-\frac{1}{2}\left(\frac{1}{d_1}+\frac{1}{d_2}\right)<1-\frac{2}{3}\frac{2d+1}{d(d+1)}=\frac{2B}{A^2}<1.
\end{equation*}

From \eqref{eq4.18}, we have
\begin{equation*}
  \frac{1}{1-\frac{2B}{A^2}}=\frac{3}{4}d+\frac{3}{8}-\frac{3}{8(2d+1)},
\end{equation*}
and combining with \eqref{eq4.1}, we obtain
\begin{equation*}
 \frac{1}{1-2S(\Omega_1)}+\frac{1}{1-2S(\Omega_2)}>\left(\frac{3}{4}d_1+\frac{3}{16}\right)+\left(\frac{3}{4}d_1+\frac{3}{16}\right)=\frac{3}{4}d+\frac{3}{8}>\frac{1}{1-\frac{2B}{A^2}}.
\end{equation*}
This completes the proof.
\end{proof}

\begin{proof}[The proof of Theorem \ref{Th:1}]

It follows from \eqref{eq2.39} that the coefficient $a_2$ of the
expansion of the function $\varepsilon_{\alpha}$ associated to
$((\Omega_1\times\Omega_2)^{{\mathbb{B}}^{d_0}}(\mu_1,\mu_2),
g(\mu_1,\mu_2)) $ is constant if and only if
\begin{equation}\label{eq4.20}
    \frac{D^{d-1}\widetilde{\chi}(d)}{(d-1)!}=
    \frac{D^{d-2}\widetilde{\chi}(d)}{(d-2)!}=0.
\end{equation}
From \eqref{eq2.61},  \eqref{eq2.62} and \eqref{eq2.63}, we get that
$a_2$ is constant if and only if
\begin{equation}\label{eq4.21}
  \left\{\begin{array}{l}
          \frac{d_1p_1}{\mu_1}+\frac{d_2p_2}{\mu_2}=d(d+1), \\
           2\frac{R(\Omega_1)}{\mu_1^2}+\frac{d_1d_2p_1p_2}{\mu_1\mu_2}+2\frac{R(\Omega_2)}{\mu_2^2}=d(d+1)\left\{(d-1)(d+2)-\frac{1}{6}(d-1)(3d+10)\right\}.
         \end{array}
  \right.
\end{equation}

Let \begin{equation}\label{eq4.22}
   x_1=\frac{d_1p_1}{A\mu_1},\;\;x_2=\frac{d_2p_2}{A\mu_2}.
\end{equation}
Then there exist positive solutions $\mu_1, \mu_2$ for
\eqref{eq4.21} if only and if there exist positive solutions $x_1,
x_2$ for \eqref{eq4.23} below.
\begin{equation}\label{eq4.23}
  \left\{\begin{array}{l}
           x_1+x_2=1,\\
           S(\Omega_1)x_1^2+x_1x_2+S(\Omega_2)x_2^2=\frac{B}{A^2}.
         \end{array}
  \right.
\end{equation}

Solving the system of equations \eqref{eq4.23}, we have
\begin{equation}\label{eq4.24}
  \left\{\begin{array}{ll}
           x_1= & \frac{1-2S(\Omega_2)+\sqrt{\Delta}}{2(1-S(\Omega_1)-S(\Omega_2))},\\
           x_2= & \frac{1-2S(\Omega_1)-\sqrt{\Delta}}{2(1-S(\Omega_1)-S(\Omega_2))},
         \end{array}
  \right.
\end{equation}
or
\begin{equation}\label{eq4.25}
  \left\{\begin{array}{ll}
           x_1= &  \frac{1-2S(\Omega_2)-\sqrt{\Delta}}{2(1-S(\Omega_1)-S(\Omega_2))},\\
           x_2= &  \frac{1-2S(\Omega_1)+\sqrt{\Delta}}{2(1-S(\Omega_1)-S(\Omega_2))},
         \end{array}
  \right.
\end{equation}
where
\begin{eqnarray}
\nonumber  \Delta &=& (1-2S(\Omega_1))^2-4(1-S(\Omega_1)-S(\Omega_2))\left(\frac{B}{A^2}-S(\Omega_1)\right) \\
\nonumber   &=& (1-2S(\Omega_2))^2-4(1-S(\Omega_1)-S(\Omega_2))\left(\frac{B}{A^2}-S(\Omega_2)\right) \\
\label{eq4.26}   &=&1-\frac{4B}{A^2}(1-S(\Omega_1)-S(\Omega_2)) -
  4S(\Omega_1)S(\Omega_2).
\end{eqnarray}
From \eqref{eq4.1} and \eqref{eq4.16}, we have that
$1-2S(\Omega_1)>0$, $1-2S(\Omega_2)>0$ and
$1-S(\Omega_1)-S(\Omega_2)>0$. So there exist positive numbers $x_1,
x_2$ satisfying \eqref{eq4.23} if only and if
\begin{equation}\label{eq4.27}
 \Delta\geq 0, \;\; 1-2S(\Omega_1)-\sqrt{\Delta}>0 \;\; ~\text{or}\;\; \Delta\geq 0, \;\;1-2S(\Omega_2)-\sqrt{\Delta}>0.
\end{equation}
That is
\begin{equation}\label{eq4.28}
 \frac{1}{1-2S(\Omega_1)}+\frac{1}{1-2S(\Omega_2)}\geq \frac{1}{1-\frac{2B}{A^2}}\;\;\;\mbox{and}\;\;\; S(\Omega_1)<\frac{B}{A^2}
\end{equation}
or
\begin{equation}\label{eq4.29}
 \frac{1}{1-2S(\Omega_1)}+\frac{1}{1-2S(\Omega_2)}\geq \frac{1}{1-\frac{2B}{A^2}}\;\;\;\mbox{and}\;\;\;  S(\Omega_2)<\frac{B}{A^2}.
\end{equation}

By \eqref{eq4.16} and \eqref{eq4.17}, we get \eqref{eq4.28} and
\eqref{eq4.29}. Thus there exist positive numbers $\mu_1, \mu_2$
such that  \eqref{eq4.21} holds. Therefore there exist positive
numbers $\mu_1, \mu_2$ such that the coefficient $a_2$ of the
Rawnsley's $\varepsilon$-function expansion of
$\left((\Omega_1\times\Omega_2)^{{\mathbb{B}}^{d_0}}(\mu_1,\mu_2),
g(\mu_1,\mu_2)\right)$ is a constant.
\end{proof}

\noindent\textbf{Acknowledgments}\quad The first author was
supported by the Scientific Research Fund of Sichuan Provincial
Education Department (No.11ZA156), and the second author was
supported by the National Natural Science Foundation of China
(No.11271291).

%%%%%%%%%%%%%%%%%%%%%%%%%%%%%%%%%%%%%%%%%%%%%%%%%%%%%%%%%%%%%%%%
%  ²Î¿¼ÎÄÏ×
%%%%%%%%%%%%%%%%%%%%%%%%%%%%%%%%%%%%%%%%%%%%%%%%%%%%%%%%%%%%%%%%
\addcontentsline{toc}{section}{References}
\phantomsection
\renewcommand\refname{References}
\small{
}

%%%%%%%%%%%%%%%%%%%%%%%%%%%%%%%%%%%%%%%%%%%%%%%%%%%%%%%%%%%%%%%%
\clearpage

\begin{thebibliography}{99}
\setlength{\parskip}{0pt}
\bibitem{Are-Loi} Arezzo C., Loi A.: Moment maps, scalar curvature and quantization of K¡§ahler manifolds, Comm. Math. Phys. \textbf{243}, 543-559 (2004)

\bibitem{AP}Ahn, H., Park, J.D.: The explicit forms and zeros of the Bergman kernel function for Hartogs type domains. Journal of Functional Analysis
\textbf{262}(8), 3518-3547 (2012)

\bibitem{Berezin}Berezin, F. A.: Quantization, Math. USSR Izvestiya \textbf{8}, 1109-1163 (1974)

\bibitem{CGR}Cahen, M., Gutt, S., Rawnsley, J.: Quantization of K\"{a}hler manifolds. I: Geometric interpretation of Berezin's quantization.
J. Geom. Phys. \textbf{7}, 45-62 (1990)

\bibitem{Cat}Catlin, D.: The Bergman kernel and a theorem of Tian.
Analysis and geometry in several complex variables (Katata, 1997),
Trends Math., Birkh\"{a}user Boston, Boston, MA, pp. 1-23 (1999)

\bibitem{Cuc-Loi}Cuccu F., Loi A.: Balanced metrics on $\mathbb{C}^n$. Journal of Geometry and Physics \textbf{57}(4), 1115-1123(2007)


\bibitem{Donaldson}Donaldson, S.: Scalar curvature and projective
embeddings, I.  J. Differential Geom. \textbf{59}, 479-522 (2001)

\bibitem{E0}Engli\v{s}, M.: Berezin Quantization and Reproducing Kernels on Complex Domains,
Trans. Amer. Math. Soc.  \textbf{348}, 411-479 (1996)

\bibitem{E1}Engli\v{s}, M.: A Forelli-Rudin construction and asymptotics of weighted Bergman kernels. J. Funct. Anal. \textbf{177}, 257-281 (2000)

\bibitem{E2}Engli\v{s}, M.: The asymptotics of a Laplace integral on a K\"{a}hler manifold. J. Reine Angew. Math. \textbf{528}, 1-39 (2000)
\bibitem{E3}Engli\v{s}, M.: Weighted Bergman kernels and balanced metrics.  RIMS Kokyuroku \textbf{1487}, 40-54(2006)


\bibitem{FK}Faraut, J., Kor\'{a}nyi, A.: Function spaces and reproducing kernels on bounded symmetric domains. J. Funct. Anal. \textbf{88}, 64-89 (1990)

\bibitem{FKKLR}Faraut, J., Kaneyuki, S., Kor\'{a}nyi, A., Lu, Q.K., Roos, G.: Analysis and Geometry on Complex Homogeneous Domains.
Progress in mathematics, Vol. 185, Birkh\"{a}user, Boston (2000)

\bibitem{Far-Tho}Faraut, J., Thomas, E.G.F.: Invariant Hilbert spaces of holomorphic functions. J. Lie Theory \textbf{9}, 383-402 (1999)

\bibitem{F}Feng, Z.M.: Hilbert spaces of holomorphic functions on generalized Cartan-Hartogs domains. Complex Variables and Elliptic Equations:
An International Journal \textbf{58}(3), 431-450 (2013)
 \bibitem{FT}Feng, Z.M., Tu, Z.H.: On canonical metrics on Cartan-Hartogs domains. Math. Z. \textbf{278}, 301-320 (2014)

\bibitem {Gre-Loi}Greco, A.,  Loi, A.: Radial balanced metrics on the unit disk. Journal of Geometry and Physics \textbf{60}(1), 53-59 (2010)

\bibitem{Hua}Hua, L.K.:  Harmonic Analysis of Functions of Several Complex Variables in the Classical Domains.
 Amer. Math. Soc., Providence, RI (1963)


\bibitem{Ji} Ji, S.: Inequality of distortion function for
invertible sheaves on abelian varieties. Duke Math. J. \textbf{58},
657-667 (1989)

\bibitem{Ke} Kempf, G.R.: Metrics on invertible sheaves on abelian varieties.
Topics in algebraic geometry (Guanajuato, 1989), Aportaciones Mat.
Notas Investigacion 5, Soc. Mat. Mexicana, Mexico, 107-108 (1992).


\bibitem{Ko} Kor$\acute{a}$nyi, A.: The volume of symmetric domains, the Koecher gamma function and an integral of Selberg.
Studia Sci. Math. Hungar. \textbf{17}, 129-133 (1982)

\bibitem{L2}Ligocka E.: On the Forelli-Rudin construction and weighted Bergman
projections, Studia Math. \textbf{94}(3), 257-272 (1989)

\bibitem{Lo}Loi, A.: The Tian-Yau-Zelditch asymptotic expansion for
real analytic K\"{a}hler metrics. Int. J Geom. Methods in Modern
Phys. \textbf{1}, 253-263 (2004)

\bibitem{Loi2}Loi, A.: Bergman and balanced metrics on complex manifolds. Int. J. Geom. Methods Mod. Phys. \textbf{02}, 553 (2005)

\bibitem{Loi-Mossa}Loi, A.,  Mossa, R.: Berezin quantization of homogeneous bounded domains. Geometriae Dedicata, \textbf{161}(1), 119-128 (2012)

\bibitem{Loi-Zed}Loi, A., Zedda, M.: Balanced metrics on Hartogs domains.  Abhandlungen aus dem Mathematischen Seminar der Universit\"{a}t Hamburg \textbf{81}(1), 69-77 (2011)

\bibitem{LZ}Loi, A., Zedda, M.: Balanced metrics on Cartan and Cartan-Hartogs domains. Math. Z. \textbf{270}, 1077-1087 (2012)

\bibitem{Loi-Zed-Zud} Loi A., Zedda M.,  Zuddas F.: Some remarks on the K\"{a}hler geometry of the Taub-NUT metrics. Annals of Global Analysis and Geometry
 \textbf{41}(4), 515-533 (2012)



\bibitem{Lu}Lu, Z.: On the lower order terms of the asymptotic
expansion of Tian-Yau-Zelditch. Amer. J. Math. \textbf{122}(2),
235-273 (2000)

\bibitem{Lui} Lui\'c, S.: Balanced Metrics and Noncommutative K\"ahler Geometry.
Symmetry, Integrability and Geometry: Methods and Applications \textbf{6}, 069, 15 pages (2010)


\bibitem{MM07}Ma, X. and  Marinescu, G.: Holomorphic Morse
inequalities and Bergman kernels. Progress in Mathematics, Vol. 254,
Birkh\v{a}user Boston Inc., Boston, MA (2007)

\bibitem{MM08}Ma, X.
and  Marinescu, G.: Generalized Bergman kernels on symplectic
manifolds. Adv. Math. \textbf{217}(4), 1756-1815 (2008)

\bibitem{MM12}Ma, X. and Marinescu, G.: Berezin-Toeplitz
quantization on K\"{a}hler manifolds. J. reine angew. Math.
\textbf{662}, 1-56 (2012)

\bibitem{Mabuchi}Mabuchi, T.:  Stability of extremal K\"ahler manifolds.
Osaka J. Math. \textbf{41}, 563-582 (2004)

\bibitem{Ra} Rawnsley, J.: Coherent states and K\"{a}hler manifolds. Q. J. Math. Oxford \textbf{28}(2), 403-415 (1977)

\bibitem{R}Rudin, W.: Function Theory in the Unit Ball of
$\mathbb{C}^n$,  Springer-Verlag (1980)


\bibitem{TW}Tu, Z.H., Wang, L.: Rigidity of proper holomorphic mappings between equidimensional Hua
domains, Math. Ann. DOI: 10.1007/s00208-014-1136-1, arXiv:
1411.3162.

\bibitem{WH}Wang A., Hao Y.: The explicit solutions for a class of complex Monge-Amp\`{e}re equations. Nonlinear Analysis: Theory, Methods $\&$ Applications \textbf{95}, 639-649 (2014)

\bibitem{X}Xu, H.: A closed formula for the asymptotic expansion of the Bergman kernel. Commun. Math. Phys. \textbf{314}, 555-585 (2012)

\bibitem{YLR}Yin, W.P., Lu, K.P., Roos, G.: New classes of domains with explicit Bergman kernel. Science in China, Series A \textbf{47}, 352-371 (2004)

\bibitem{YW}Yin, W.P., Wang,  A.: The equivalence on classical metrics. Science in China Series A: Mathematics \textbf{50}(2), 183-200 (2007)

\bibitem{Zed}Zedda, M.: Canonical metrics on Cartan-Hartogs domains. International Journal of Geometric Methods in Modern Physics \textbf{9}(1), 1250011 (13 pages) (2012)

\bibitem{Zeld}Zelditch, S.: Szeg\"{o} kernels and a theorem of
Tian. Internat. Math. Res. Notices \textbf{6}, 317-331 (1998)

\bibitem{Zhang} Zhang, S.: Heights and reductions of semi-stable varieties,
Compositio Math. \textbf{104}(1), 77-105 (1996)
\end{thebibliography}
\end{document}